\documentclass[12pt]{article}

\usepackage{mathtools}
\usepackage{amsthm}
\usepackage{graphicx}
\usepackage[export]{adjustbox}
\usepackage{xspace}

\usepackage{tikz}
\usepackage{amsfonts}
\usetikzlibrary{arrows.meta}
\usetikzlibrary{arrows}
\usetikzlibrary{decorations.markings}
\usetikzlibrary{math} %needed tikz library

\usepackage{comment}

\newcommand{\Myarrow}{\mathbin{\tikz[baseline] \draw [ -{Straight Barb}] ++(0, 0)  ;}}
\newcommand{\MyarrowS}{\mathbin{\tikz[baseline] \draw [ -{Stealth[open]}] ++(0, 0)  ;}}

\usepackage{xint}
\usepackage{xintfrac}
\usepackage{xinttools}
\usepackage{xintexpr}
\usepackage{etoolbox}
\usepackage{comment}

\usepackage{array}

\usepackage{jabbrv}

\usepackage{cleveref}
\crefname{figure}{Figure}{Figures}
\Crefname{figure}{Figure}{Figures}

\newtheorem{theorem}{Theorem}
\newtheorem{lemma}{Lemma}
\newtheorem{coro}{Corollary}[theorem]
\newtheorem{prop}{Proposition}
\newtheorem{defi}{Definition}
\newtheorem{remark}{Remark}
\newtheorem{claim}{Claim}

\newcommand{\coe}{h}
\newcommand{\cgf}{H}
\newcommand{\cob}{b}
\newcommand{\cdgf}{\mathcal{H}}

\newcommand{\F}{F}
\newcommand{\R}{R}
\newcommand{\G}{G}
\newcommand{\Sum}{S}

\newcommand{\POne}{$\mathcal{P}_1$\xspace}
\newcommand{\PTwo}{$\mathcal{P}_2$\xspace}
\newcommand{\POnem}{\mathcal{P}_1}
\newcommand{\PTwom}{\mathcal{P}_2}
\newcommand{\e}{e}
\newcommand{\bone}{\alpha}
\newcommand{\btwo}{\beta}
\newcommand{\D}{U}

\newcommand{\ca}{\lambda}
\newcommand{\cb}{\eta}

\newcommand{\x}{x}
\newcommand{\y}{y}
\newcommand{\X}{X}
\newcommand{\Y}{Y}

\newcommand{\Go}{Gosper\xspace}
\newcommand{\Ze}{Zeilberger\xspace}
\newcommand{\anti}{antidifference\xspace}
\newcommand{\qu}{quadrant\xspace}
\newcommand{\sumtS}{w}
\newcommand{\Inhom}{g}

\newcommand{\wrt}{w.r.t.\@\xspace}

\DeclareRobustCommand{\stirling}{\genfrac\{\}{0pt}{}}

\newcolumntype{L}{>{$}l<{$}} % math-mode version of "l" column type
\newcolumntype{C}{>{$}c<{$}} % math-mode version of "c" column type
\newcolumntype{R}{>{$}r<{$}} % math-mode version of "r" column type

\usepackage{microtype}

\begin{document}

\title{Regular structures of an intractable enumeration problem: a diagonal recurrence relation of monomer-polymer coverings on two-dimensional rectangular lattices}

\author{Yong Kong\\
  Department of Biostatistics \\
  School of Public Health \\
  Yale University \\
  New Haven, CT 06520, USA \\
  email: \texttt{yong.kong@yale.edu}
}%

\date{}
\maketitle

\tableofcontents

\begin{abstract}
  In the monomer-polymer model, 
  a linear rigid polymer covers $k$ adjacent lattice sites,
  with no lattice site occupied by more than one polymer.
  The polymers are called $k$-mers,
  and those unoccupied lattice sites are called monomers.
  The well-known monomer-dimer model is a special case of the monomer-polymer model with $k=2$.
  The enumeration of polymer coverings on two-dimensional rectangular lattices is
  considered as ``intractable''.
  We prove that the number of coverings of $s$ polymer satisfies a simple recurrence relation
  $
    \sum_{i=0}^{2s} (-1)^i \binom{2s}{i} a_{n-i, m-i} = 2^s {(2s)!} / {s!}
  $
  on a $n \times m$ rectangular lattice with open boundary conditions
  in both directions.
\end{abstract}

\section{Introduction} \label{S:intro}

The main result of this paper is Theorem~\ref{T:main},
which gives a diagonal recurrence relation for the number of coverings of polymers 
on a $n \times m$ rectangular lattice with open boundary conditions
in both directions.

Enumerations of rigid rodlike molecules
on lattices has been studied by researchers 
in diverse fields of physics, mathematics, and theoretical computer science
for a long time.
In the monomer-polymer model, 
a linear rigid polymer covers $k$ adjacent lattice sites,
with no lattice site occupied by more than one polymer.
The polymers are called $k$-mers,
and those unoccupied lattice sites are called monomers.

Physicists have a long history of studying this model for phase transitions in equilibrium
statistical mechanics~\cite{Onsager1949,Ghosh2007b,Dhar2021,Rodrigues2023}.
When $k=2$, the model becomes the well-known monomer-dimer model,
which is closely related to the Ising model~\cite{Onsager1944,Kasteleyn1963,Fisher1966b,McCoy1973}.
While it has been shown that
the special case of close packed dimer problem
can be solved analytically for any \emph{planar} lattices using the Pfaffian method~\cite{Kasteleyn1963},
no solution has been found
for the general monomer-dimer problem where unoccupied lattice sites (monomers) are allowed,
despite decades of effort (for example, see ~\cite{Tzeng2003}).

This stark dichotomy in the computational complexity of the monomer-dimer problem has captivated researchers in theoretical computer science, particularly those focused on computational complexity theory.
It has been shown that the enumeration of monomer-dimer configurations in \emph{planar} lattices
is \#P-complete~\cite{Jerrum1987},
which indicates the problem is computationally ``intractable''.
For counting problems,
the class \#P plays the same role 
as the more familiar NP class does for decision problems
(such as the well-known satisfiability problem).
The \#P-complete problems are at least as hard as the NP-complete
problems in computational complexity hierarchy.
Currently 
``P versus NP'' problem is perhaps the major
outstanding problem in theoretical computer science.
The prevailing view on the P vs NP problem
is that P is not equal to NP - there are fundamental differences between these two complexity classes,
and the boundaries between ``easy'' and ``intractable'' problems are impenetrable.

Even though the problem of enumeration of monomer-polymer coverings is now considered as
``intractable'', regular structures have been found for the problem.
For example, a simple recurrence relation was found
for the enumeration of polymer coverings
on two-dimensional lattice strips~\cite{Kong2024}.
In this paper we will explore more regular structures in this problem,
and prove a simple \emph{diagonal} recurrence relation for square lattices (Theorem~\ref{T:main}).
Aesthetically and practically, the results on square lattices
are more appealing than the narrow strips.
Various optimization techniques have been used to extend the width
of the lattice strips in the past,
for example, \cite{Kong2006c,Kong2006,Kong2006b,Kong2007,Ghosh2007b},
but still the computational complexity grows exponentially with the
lattice width.
The diagonal recurrence relation,
as proven in this paper,
allows us to investigate previously uncharted regions with larger lattice sizes
that were inaccessible using earlier methods.

The main result is stated below.
Consider a two-dimensional regular lattice with a width of $n$ and a length of $m$,
with open boundary conditions in both directions.
Let $s$ denote the number of $k$-mers in the lattice,
and denote the number of configurations of the $s$ $k$-mers on the lattice
by $a_{n,m}$.
Then $a_{n,m}$ satisfies the following diagonal recurrence relation.

\begin{theorem} \label{T:main}
  The number of configurations of $s$ $k$-mers on a $n \times m$ lattice
  with open boundary conditions in both directions
  satisfies the following recurrence
\begin{equation} \label{E:2s}
  \sum_{i=0}^{2s} (-1)^i \binom{2s}{i} a_{n-i, m-i}
  = 2^s \frac{(2s)!}{s!} ,
\end{equation}
for $n \ge (k+1)s $ and $m \ge (k+1)s $.
\end{theorem}

From this recurrence we can obtain another recurrence directly.
\begin{coro}
\[
  \sum_{i=0}^{2s+1} (-1)^i \binom{2s+1}{i} a_{n-i, m-i} = 0,
\]
for $n \ge (k+1)s + 1 $ and $m \ge (k+1)s + 1 $.
\end{coro}

\begin{proof}
  Use the identity
  \[
    \binom{2s+1}{i}
    =
    \binom{2s}{i}
    +
    \binom{2s}{i-1} .
  \]
\end{proof}

The proof of Theorem~\ref{T:main} is based on the following
recurrence relation on a lattice strip~\cite{Kong2024}.
\begin{lemma}[Recurrence on a lattice strip] \label{L:strip}
  For given $k$, $n$ and $s$, the following recursive relation holds
  for a lattice strip with open boundary conditions
  in both directions.
\begin{equation} \label{E:rec1}
 \sum_{i=0}^{s} (-1)^i \binom{s}{i} a_{n, m-i} = c(n, k)^s, \qquad n \ge k \text{ and } m \ge ks,
\end{equation}
where $c(n, k) = 2n-k+1$.
\end{lemma}

\begin{remark}
  The $c(n, k)$ on rhs of Eq.~\eqref{E:rec1}
  is a constant that depends on the boundary conditions as well as $n$ and $k$,
  but not $m$ or $s$.
  For different boundary conditions $c(n, k)$ can take different
  values~\cite{Kong2024}.
  Note also that Eq.~\eqref{E:rec1} holds when $n < k$, but with a different value of
  $c(n, k)$~\cite{Kong2024}.
  In this paper we only need the conditions stated in Lemma~\ref{L:strip}.
\end{remark}

\begin{remark}
Note that the rhs of Eq.~\eqref{E:2s} depends only on $s$, the number of
$k$-mers. It does not depend on $k$, $n$, or $m$.
This is in contrast to the recurrence of lattice strips (Lemma~\ref{L:strip}),
where the constant on the rhs  depends on $s$ as well as $n$, the width of the lattice strip.
\end{remark}

The paper is organized as follows.
Since the proof of Theorem~\ref{T:main} relies crucially on a two-dimensional sequence, Section~\ref{S:seq} introduces this sequence and establishes the properties needed in the subsequent analysis.
Section~\ref{S:proof}, which constitutes the main part of the paper, is devoted to the proof of Theorem~\ref{T:main}.
The proof requires the evaluation of numerous single and double summations, for which the principal tools are the \Go-algorithm and the \Ze-algorithm.
These algorithms are summarized in Appendix~\ref{S:A_GZ}.
For the convenience of readers unfamiliar with these techniques, the appendix provides the necessary background material together with extensions for handling double sums with nonstandard boundary conditions.

A toy example for $s=4$ is given in section~\ref{SS:toy} with explicit calculations.

Throughout the paper 
we use $k$ for the length of the polymer (the $k$-mer),
$n$ and $m$ for the width and length of the lattice, respectively,
and $s$ for the number of the polymers.
The exceptions are in Appendix~\ref{S:A_GZ} on \Go{} and \Ze{} algorithms,
where $k$ is conventional used as the summation index,
and $n$ as the summation parameter.

\section{A two-dimensional sequence} \label{S:seq}

  The proof of the main theorem depends critically on a two-dimensional sequence,
  which acts as the coefficients for the recurrence identities that are
  added to specified lattice sites inside the square of size $(2s+1) \times (2s+1)$.
  This sequence will play an important role in our analysis.
  
  As the discussions below and its uses in the following proofs show,
  this sequence
  in several aspects looks like the sequence of binomial coefficients,
  but with an extra parameter $s$, the number of $k$-mers covering the lattice.
  It plays for set \PTwo\ a role analogous to that played by the binomial coefficients for set \POne.

  Below we first give a recursive definition of the sequence,
  then derive an explicit formula for the ordinary generating function of the sequence,
  from which the explicit expression of the sequence
  is obtained.

  The explicit expressions for the sequence
  have \emph{three terms}.
  As the proof of the main theorem shows,
  each term plays a distinct role
  in building up or canceling out the coefficients
  of the recurrence identities added to each position in the
  lattice square.

  In Appendix~\ref{A:seq} we show some additional interesting properties of this two-dimensional sequence.
  There we derive the double generating function of the sequence,
  and show that each element of the sequence is an integer.

\begin{defi}
  The coefficient sequence  $\coe_{s,i,j}$ is defined recursively as
\begin{equation} \label{E:S_def}
  \coe_{s,i,j} =
  \begin{dcases}
    \binom{s+j-1}{j} \frac{s-j}{s},  & i = 0, \\
    - \frac{s-j+1}{i} \coe_{s, i-1, j-1} - \frac{j-i}{i}  \coe_{s, i-1, j}  , & i > 0 , \\
    0,  & j \le i. 
  \end{dcases}  
\end{equation}
\end{defi}

As examples, 
the sequence for $s=4$ is shown in Table~\ref{T:s4},
and the sequence for $s=5$ is shown in Table~\ref{T:s5}.

\begin{table}[th]
  \centering
  \caption{
    The sequence of $\coe_{4,i,j}$.
    \label{T:s4}}
  \begin{tabular}{C|RRRR}
    \hline\hline
    i \backslash j & 1 & 2 & 3 & 4\\
    \hline
  0 &  3 & 5 & 5 & 0 \\
  1 &   &-14 & -20 & -5 \\
  2 &   &    & 24 & 15  \\
    3 &   &    &    & -13 \\
    \hline
  \end{tabular}
\end{table}

\begin{table}[th]
  \centering
  \caption{
    The sequence of $\coe_{5,i,j}$.
    \label{T:s5}}
  \begin{tabular}{C|RRRRR}
    \hline\hline
    i \backslash j & 1 & 2 & 3 & 4 & 5\\
    \hline
  0 &  4 &   9 &  14 &  14  &  0   \\
  1 &    & -25 & -55 & -70  & -14  \\
  2 &    &     &  65 & 125  &  56  \\
  3 &    &     &    &  -85  & -79  \\
  4 &    &     &    &       &  41  \\
    \hline
  \end{tabular}
\end{table}

Define the ordinary generating function for the sequence as
\[
  \cgf_{s, i} =  \sum_{j=i+1} \coe_{s, i, j} x^j .
\]
We have the following recursive relation for the generating function $\cgf_{s, i}$,
where $\cgf'$ is the derivative of $\cgf$ \wrt $x$.
Since $s$ is fixed, below we omit $s$ in $\cgf_{s,i}$ and $\coe_{s,i,j}$.
\begin{lemma} [Recursive relation of the generating function] \label{L:S_gf_recur}
  For $i>1$,
  \begin{equation} \label{E:S_gf_recur}
   \cgf_{i}  = -\left( \frac{sx}{i} -1 \right) \cgf_{i-1} + \frac{x(x-1)}{i} \cgf_{i-1}' .
 \end{equation}
\end{lemma}

\begin{proof}
  This can be proved directly from the definition of the sequence (Eq.~\eqref{E:S_def}):
\begin{align*}
  \cgf_{i} &=\sum_{j=i+1} \coe_{i, j} x^j \\
           &= -\sum_{j=i+1} \left[ \frac{s-j+1}{i} \coe_{i-1, j-1} + \frac{j-i}{i} \coe_{i-1, j}  \right] x^j  \\
           &= -\frac{(s+1)x}{i} \sum_{j=i} \coe_{i-1, j} x^{j} +  \frac{x}{i} \sum_{j=i} (j+1) \coe_{0, j} x^{j}
             - \frac{1}{i}\sum_{j=i} (j-i)\coe_{i-1, j}  x^j \\
           &= -\frac{(s+1)x}{i} \cgf_{i-1} +  \frac{x}{i} \left[ x \cgf_{i-1}' + \cgf_{i-1} \right]
             - \frac{1}{i} \left[ x \cgf_{i-1}' - i \cgf_{i-1} \right] \\
  &= -\left( \frac{sx}{i} -1 \right) \cgf_{i-1} + \frac{x(x-1)}{i} \cgf_{i-1}' .
\end{align*}
\end{proof}

From this recursive relation of the generating function,
we obtain the explicit expression
for $\cgf_{i}$, containing three terms: 
\begin{prop}[Explicit formula for $\cgf_{i}$] \label{P:S_gf}
  The explicit formula for $\cgf_{i}$ is
  \begin{align} \label{E:H}
    \cgf_{i} &=
               \sum_{j=0}^i (-1)^{j+1}  \binom{s}{j} x^j \nonumber \\
             & + (-1)^{i+1}  \binom{2s}{ i}  \frac{ x^{i+1} }{ (1-x)^{s+1} } \nonumber \\
             & + (-1)^{i} (2s-i) \binom{2s}{i}
               \sum_{t=0}^{i} (-1)^t \binom{i}{t}
               \frac{ 1 } {2s - t}
               \frac{1}{ (1-x)^{s-t}} .
  \end{align}
\end{prop}

\begin{proof}
  The formula can be proved by induction on $i$ using the recurrence Eq.~\eqref{E:S_gf_recur}.
  When $i=0$,
  the coefficient $\coe_{0, j}$ can be written as two parts,
  \[
    \coe_{0, j} = \binom{s+j-1}{j}  - \frac{j}{s}  \binom{s+j-1}{j}.
  \]
  Since
  \[
    \sum_{j=1} \binom{s+j-1}{j} x^j = -1 + \frac{1}{(1-x)^s},
  \]
  the generating function $\cgf_{0}$ can be calculated as
  \begin{align*}
    \cgf_{0} &= -1 + \frac{1}{(1-x)^s} - \frac{1}{s} x D_x \left[  \frac{1}{(1-x)^s} \right] \\
    &= -1 - \frac{1}{(1-x)^{s+1}} + \frac{1}{(1-x)^{s}},
  \end{align*}
  which agrees with Eq.~\eqref{E:H} term by term.
  Here $D_x$ is the differentiation operator $d/dx$.

  For $\cgf_{i+1}$, we consider the three terms separately
  and prove that each of them satisfies the recurrence of Eq.~\eqref{E:S_gf_recur}.
  
  For the first term of $\cgf_{i+1}$,  the recursive relation of Eq.~\eqref{E:S_gf_recur}
  gives
  \begin{align}
    &
                       -\left( \frac{sx}{i+1} -1 \right) \cgf^{(1)}_{i}
                       + \frac{x(x-1)}{i+1} \cgf^{(1)\prime}_{i}   \nonumber \\
    =&  -\left( \frac{sx}{i+1} -1 \right) \sum_{j=0}^i (-1)^{j+1}  \binom{s}{j} x^j
       + \frac{x-1}{i+1} \sum_{j=0}^i (-1)^{j+1} j \binom{s}{j} x^{j}  .  \label{E:seq_h1}
  \end{align}
  The second term in Eq.~\eqref{E:seq_h1} can be transformed to
  \begin{align*}
    & \frac{x}{i+1}  \sum_{j=0}^i (-1)^{j+1} j \binom{s}{j} x^{j}
      - \frac{x}{i+1} \sum_{j=0}^i (-1)^{j+1} j \binom{s}{j} x^{j-1} \\
    =& \frac{x}{i+1}  \sum_{j=0}^i (-1)^{j+1} j \binom{s}{j} x^{j}
       + \frac{x}{i+1} \sum_{j=0}^{i-1} (-1)^{j+1} (j+1) \binom{s}{j+1} x^{j} \\
    =& \frac{s x}{i+1} \sum_{j=0}^i (-1)^{j+1} \binom{s}{j} x^{j}
       + (-1)^i \binom{s}{i+1} x^{i+1} ,
  \end{align*}
  where the identity $j \binom{s}{j} + (j+1)\binom{s}{j+1} = s \binom{s}{j} $
  is used in the second step to combine the two sums.
  The above expression, when combined with the first term of Eq.~\eqref{E:seq_h1},
  shows it equals to $\cgf^{(1)}_{i+1}$
  \[
    \cgf^{(1)}_{i+1} = \sum_{j=0}^{i+1} (-1)^{j+1}  \binom{s}{j} x^j .
  \]

  For the second term of $\cgf_{i+1}$,
  after factoring out common factors,
  \begin{align*}
    &
      -\left( \frac{sx}{i+1} -1 \right) \cgf^{(2)}_{i}
      + \frac{x(x-1)}{i+1} \cgf^{(2)\prime}_{i}   \nonumber \\  
    =&
       (-1)^{i+1}  \binom{2s}{ i}  \frac{ x^{i} }{ (1-x)^{s+1} }
       \left[
       \left(
       \frac{sx}{i+1} -1
       \right)
       x
       + \frac{x(x-1)}{i+1}
       \left(
       i+1 + \frac{(s+1)x}{1-x}
       \right)
       \right]\\
    =& (-1)^{i+2} \binom{2s}{ i+1 }  \frac{ x^{i+2} }{ (1-x)^{s+1} },
  \end{align*}
  which equals to $\cgf^{(2)}_{i+1}$ .

  For the third term of $\cgf_{i+1}$,
  consider the following two equations:
  \begin{align}
     &
      -\left( \frac{sx}{i+1} -1 \right) \cgf^{(3)}_{i}
       + \frac{x(x-1)}{i+1} \cgf^{(3)\prime}_{i} = \sum_{t=0}^{i} F_1  , \label{E:seq_h3_F1} \\
     & \begin{multlined}
       (-1)^{i+1} (2s-i-1) \binom{2s}{i+1}
               \sum_{t=0}^{i+1} (-1)^t \binom{i+1}{t}
               \frac{ 1 } {2s - t}
       \frac{1}{ (1-x)^{s-t}}
       =
       \sum_{t=0}^{i} F_2 \\
       +
       \binom{2s} {i+1}
       \frac{1} { (1-x)^{s-i-1} } ,
       \end{multlined}
       \label{E:seq_h3_F2}
  \end{align}
  where the summands are
  \begin{align*}
    F_1 &= 
          (-1)^{i+t}   (2s-i)
          \binom{2s}{ i}
          \binom{i} {t}
          \frac{1} {2s-t}
          \frac{1} {(1-x)^{s-t} }
          \left[
          -\left(
          \frac{s x} {i+1} - 1
          \right)
          - \frac{x (s-t) }{i+1}
          \right], \\
    F_2 &=
          (-1)^{i+t+1}
          (2s-i-1)
          \binom{2s}{ i+1}
          \binom{i+1} {t}
          \frac{1} {2s-t}
          \frac{1} {(1-x)^{s-t} } .
          \end{align*}

          The sum with summand $F_1$ in Eq.~\eqref{E:seq_h3_F1} is the result of substituting 
          $\cgf^{(3)}_{i}$ into the recurrence Eq.~\eqref{E:S_gf_recur}.
          The left side of Eq.~\eqref{E:seq_h3_F2}
          is in the form of $\cgf^{(3)}_{i+1}$
          according to Eq.~\eqref{E:H},
          and the the right side is 
          its rearrangement.
          We will demonstrate that the right sides of these two equations are equal,
          thereby proving the correctness of $\cgf^{(3)}_{i+1}$.

          The difference of $F_1$ and $F_2$
          is \Go-summable (see Appendix~\ref{S:A_GZ}),
          with the \anti
          \[
            (-1)^{i+t+1}
            \binom{2s} {i+1}
            \binom{i} {t-1}
            \frac{1} {(1-x)^{s-t} } .
          \]
          Hence the sum of $F_1 - F_2$ is
          \[
            \sum_{t=0}^{i} (F_1 - F_2) =
            \binom{2s} {i+1}
            \frac{1} { (1-x)^{s-i-1} },
          \]
          which is the second term in Eq.~\eqref{E:seq_h3_F2}.
          This shows that the right sides of Eq.~\eqref{E:seq_h3_F1} and Eq.~\eqref{E:seq_h3_F2}
          are equal.
\end{proof}

From the explicit expression of the generating function Eq.~\eqref{E:H},
the explicit expression of the sequence can be obtained
by extracting the coefficient of $x$ from $\cgf_i(x)$.
\begin{coro}[Explicit formula for $\coe_{i, j}$] \label{C:hij}
  For $i \ge 0$,
  \begin{align} \label{E:hv2}
    \coe_{i,j} &= 
                 (-1)^{j+1}  \binom{s}{j} [j \le i] \\
               & + (-1)^{i+1} \binom{2s}{i} \binom{s+j-i-1}{s}                    \notag \\
               & + (-1)^{i} (2s-i) \binom{2s}{i} \sum_{t=0}^{i} (-1)^t \frac{1}{2s - t}
                 \binom{i}{t}  \binom{s-t-1+j}{ j},  \notag 
  \end{align}  
  where $[\cdot]$ in the first term is Iverson's convention:
  $[\ell]=1$ if $\ell$ is true, otherwise  $[\ell]=0$~\cite{Graham1994}.
This term is to make $\coe_{i,j} = 0$ 
for $j \le i$.
\end{coro}
We denote the three terms of $\coe_{i,j}$ by $\coe_{i,j}^{(1)}$, $\coe_{i,j}^{(2)}$, and $\coe_{i,j}^{(3)}$:
\begin{align*}
  \coe_{i,j}^{(1)} &= (-1)^{j+1}  \binom{s}{j} [j \le i], \\
  \coe_{i,j}^{(2)} &=  (-1)^{i+1} \binom{2s}{i} \binom{s+j-i-1}{s}, \\
  \coe_{i,j}^{(3)} &= (-1)^{i} (2s-i) \binom{2s}{i} \sum_{t=0}^{i} (-1)^t \frac{1}{2s - t}
                 \binom{i}{t}  \binom{s-t-1+j}{ j}.
\end{align*}

\section{Preliminaries} \label{S:pre}

\subsection{Outline of the proof}

The two main ingredients of the proof are the recurrence identities
on the \emph{lattice strips},
where the recurrences are on one of the indices of $a_{n,m}$ (Lemma~\ref{L:strip})~\cite{Kong2024},
and the sequence $\coe_{i, j}$ discussed in the previous section (Eq.~\eqref{E:hv2}
of Corollary~\ref{C:hij}).

Consider a $(2s+1) \times (2s+1)$ square inside the two-dimensional lattice, with
the top-right coordinate as $(n, m)$.
To clarify the terminology,
we call the diagonal from lower-left corner to the upper-right corner
as the \emph{main diagonal},
or simply \emph{diagonal},
instead of anti-diagonal in the context of matrices.
Inside this square,
by choosing appropriate weights,
we add the recurrence identities of the lattice strips (Eq.~\eqref{E:rec1}),
in both horizontal and vertical directions symmetrically.
After summing up these identities,
along the diagonal of this square
the coefficients of $a_{n-i, m-i}$ will be twice that given in
Theorem~\ref{T:main},
while at the off-diagonal positions
the coefficients of $a_{n-i, m-j}$ will vanish at each lattice site for $i \ne j$.
Meanwhile, the sum of the right hand sides of these recurrences
gives twice that of the right hand of Eq.~\eqref{E:2s}.

Since the lattice have free boundary conditions in both directions,
the recurrence of the lattice strips of Lemma~\ref{L:strip}
applies to both horizontal and vertical directions.
For the vertical direction, the recurrence is
  \begin{equation} \label{E:rec_v}
   \e^v_{n,m} :=  \sum_{j=0}^{s} (-1)^j \binom{s}{j} a_{n, m-j} = c(n, k)^s, \quad m \ge ks,
  \end{equation}
and for the horizontal direction, the recurrence is
  \begin{equation} \label{E:rec_h}
   \e^h_{n,m} := \sum_{i=0}^{s} (-1)^i \binom{s}{i} a_{n-i, m} = c(m, k)^s, \quad n \ge ks.
  \end{equation}

The top-right corner of the square that we are focusing on has the coordinates $(n, m)$.
In general, lattice sites inside the $(2s+1) \times (2s+1)$ square have coordinates in the form
of $(n-i, m-j)$, for $0 \le i, j \le 2s$.
Below to save space we abbreviate the coordinates as $(i, j)$ when there is no ambiguity.
Hence the top-right corner is labeled as $(0,0)$,
and the bottom-left corner $(2s, 2s)$.

We use the largest coordinates (top right) in the recurrence relations Eq.~\eqref{E:rec_v}
and Eq.~\eqref{E:rec_h} to label the identities,
and by using abbreviation just mentioned,
at the lattice site $(i, j)$
the recurrence relations are
\begin{align*}
  \e^v_{i,j} &:=  \sum_{j'=0}^{s} (-1)^{j\prime} \binom{s}{j'} a_{n-i, m-j-j'} = c(n-i, k)^s , \\
  \e^h_{i,j} &:= \sum_{i'=0}^{s} (-1)^{i\prime} \binom{s}{i'} a_{n-i-i', m-j} = c(m-j, k)^s  .         
\end{align*}

To facilitate the discussions, we divide the necessary recurrence identities of the lattice strips
used in the proof into two sets,
\POne and \PTwo,
and denote the corresponding weights at lattice site $(i,j)$
by $\cob_{i,j}$ and
$\bar{\coe}_{i,j}$, respectively.
The details of \POne and \PTwo are discussed in the following subsections.
When the recurrence identities in \POne or \PTwo are summed up,
they introduce terms such as $a_{n-i, m-j}$.
We use $\bone_{i,j}$ and $\btwo_{i,j}$ to denote the coefficient of these terms
at position $(i, j)$.
\begin{align*}
  \bone_{i,j} &= \text{coefficient of $a_{n-i, m-j}$} \, \text{in}
                \left(
                \sum_{\e^v_{i',j'}  \in \POnem} \cob_{i',j'} \e^v_{i',j'}
                +
                \sum_{\e^h_{i',j'} \in \POnem} \cob_{i',j'} \e^h_{i',j'}
                \right), \\
  \btwo_{i,j} &= \text{coefficient of $a_{n-i, m-j}$} \, \text{in}
                \left(
                \sum_{\e^v_{i',j'}  \in \PTwom} \bar{\coe}_{i',j'} \e^v_{i',j'}
                +
                \sum_{\e^h_{i',j'}  \in \PTwom} \bar{\coe}_{i',j'} \e^h_{i',j'}
                \right),
\end{align*}
where $\bar{\coe}_{i,j}$ is related to the two-dimensional sequence discussed
in section~\ref{S:seq} (see below).

Our goals are to prove, for the left hand side of all these added up identities
of lattice strips,
\begin{align} \label{E:goal_lhs}
  \begin{aligned}
  \bone_{i,i} + \btwo_{i,i} &= (-1)^i \binom{2s}{i} \times 2,\\
  \bone_{i,j} + \btwo_{i,j} &= 0, \quad i \ne j,
  \end{aligned}
\end{align}
and for the right hand side,
\begin{equation} \label{E:goal_rhs}
  \sum_{\e^v_{i',j'}  \in \POnem} \cob_{i',j'} c(n-i', k)^s
  +
  \sum_{\e^v_{i',j'}  \in \PTwom} \bar{\coe}_{i',j'} c(n-i', k)^s
  = 2^s \frac{(2s)!}{s!}.
\end{equation}
% check this x 2? Not in the original - should be correct: only vertical here.

\subsection{The two sets of recurrence identities: \POne and \PTwo} \label{SS:P1P2}

In this section we describe how to add the lattice strip recurrences
in the vertical and horizontal directions with appropriate coefficients,
so that the terms in the diagonal have the correct values of the main theorem,
while the left side of the off-diagonal terms vanish (Eq.~\eqref{E:goal_lhs}).

We introduce two sets of recurrence identities, \POne and \PTwo.
For the recurrence identities in the set \POne,
binomial coefficients are used as weights,
while for the set \PTwo the two-dimensional sequence introduced above in section~\ref{S:seq}
is used as weights.

We will prove the main result quadrant by quadrant
inside the $(2s + 1) \times (2s + 1)$ square.
The coordinate of the center lattice site is $(s, s)$.
The four quadrants are naturally defined by
their four corners (counterclockwise):
\begin{align*}
  Q_1 &= \{(s, s), (0, s), (0, 0), (s, 0)\},     \quad \text{top right},  \\
  Q_2 &= \{(s, s), (s, 0), (2s, 0), (2s, s)\},   \quad \text{top left},   \\
  Q_3 &= \{(s, s), (2s, s), (2s, 2s), (s, 2s)\}, \quad \text{bottom left}, \\
  Q_4 &= \{(s, s), (s, 2s), (0, 2s), (0, s)\},   \quad \text{bottom right}.
\end{align*}
  
To clarify the discussion, Figures~\ref{F:S1} and \ref{F:S2} show
how the recurrences on a strip (Eqs.~\eqref{E:rec_v} and ~\eqref{E:rec_h})
  are added at each lattice site for \POne and \PTwo respectively with $s=4$.
  Similarly,  Figures~\ref{F:S1s5} and \ref{F:S2s5} for $s=5$.
  The location of the arrowhead indicates the highest index in the recurrence (top right),
  %the largest coordinates (top right) of the recurrences,
  while direction of the arrowheads show in which direction (horizontal or vertical)
  the recurrences are applied:
  if the arrowhead points downwards (leftwards),
  the recurrence on the vertical (horizontal) strip is applied.
  Note the original coordinates (not the abbreviated ones) are used in the figures.
  For example, at $(n-1, m-1)$ in Figure~\ref{F:S1} for $s=4$,
  the blue arrowhead pointing to the left indicates that
  the horizontal recurrence Eq.~\eqref{E:rec_h} is applied at that location
  for a horizontal strip to the left.
  The adjacent number $-4$ is the weight for the recurrence.
  Similarly the red arrowhead pointing downwards at the same coordinates
  indicates that the vertical recurrence Eq.~\eqref{E:rec_v} is applied at that location,
  with the same weight $-4$.
  
  These figures also show the lattice sites affected by the added recurrences.
  Since the number of terms in the strip recurrences is $s+1$ (Eq.~\eqref{E:rec1}),
  the figures shows that affected lattice sites are restricted to the
  $(2s + 1) \times (2s + 1)$ square.
  For the above mentioned example of the horizontal recurrence added at $(n-1, m-1)$ for $s=4$ in Figure~\ref{F:S1},
  the affected lattice sites are from $(n-1, m-1)$ to $(n-5, m-1)$.

  As we can see from \cref{F:S1,F:S1s5,F:S2,F:S2s5},
  the lattice sites in different quadrants are affected differently
  by the added vertical and horizontal recurrence
  identities. We will treat each quadrant separately.

\subsubsection{The set of recurrence identities in \POne} \label{SSS:P1}

In the set of \POne{},
we put the following vertical recurrence relations at lattice site $(i, j)$:
\begin{equation*} \label{E:P1_v}
  (-1)^j \binom{s}{j} \e^v_{i,j},  
\end{equation*}
for $0 \le i \le s$, $0 \le j \le i$ in the first \qu,
and
$s < i \le 2s$, $ i-s \le j \le s$
in the second \qu.
Similarly, horizontal recurrence relations
are put in the set of \POne{}
with the same coefficients in a symmetrical way with respect to the main diagonal.
In other word, the weights $\cob^v_{i,j}$ and $\cob^h_{i,j}$
(the superscripts $v$ and $h$ stand for vertical and horizontal respectively) are
\begin{equation} \label{E:P1}
  \begin{aligned}
  \cob^v_{i,j} &=  (-1)^j \binom{s}{j} , \\
  \cob^h_{i,j} &=  (-1)^i \binom{s}{i} .                
\end{aligned}
\end{equation}

Figure~\ref{F:S1} schematically shows the recurrence relation
identities in \POne{} for $s=4$,
and Figure~\ref{F:S1s5} for $s=5$.
The numbers at the lattice sites are the coefficients (weights) $\cob_{i,j}$
for the recurrence identity starting at that lattice site.

\begin{comment}
The red arrowheads at the position $(i, j)$, which point downwards,
indicate the vertical recurrence identity
$\e^v_{i,j}$,
while the blue arrowheads, which point to the left,
are for the horizontal recurrence identity
$\e^h_{i,j}$.
The numbers at the lattice sites are the coefficients (weights) $\cob_{i,j}$
for the recurrence identity starting at that lattice site.

Note that since the length of the recurrence relation
of lattice strips is $s+1$,
by construction the terms introduced by the identities in the set \POne{}
are restricted inside the  $(2s+1) \times (2s+1)$ square.
\end{comment}

As shown below,
the summation of the recurrence identities in \POne{}
gives twice the coefficients of Eq.~\eqref{E:2s} along the diagonal,
but it also introduces nonzero entries of the off-diagonal terms
($a_{n-i, m-j}, \text{ where } i \ne j$).
The set of recurrence identities in \PTwo (described below) is used to
cancel these off-diagonal terms.

  For the identities of vertical recurrences in the set \POne{}, 
  when $0 \le i \le s$,
  the affected sites are in the range of $0 \le j \le s+i$; and when
  $s < i \le 2s$,
  the affected sites are in the range of $i-s \le j \le 2s$.
  By simple variable changes we see that the coefficient of  $a_{n-i, m-j}$
  induced by the vertical recurrences in the set \POne{} is given by
  \begin{equation}  \label{E:down}
    \bone_{i, j}^{v} =
    (-1)^j  \sum_{j'=\max(0, i-s,j-s)}^{ \min(s, i, j)  }  \binom{s}{j'}  \binom{s}{j-j'} .
  \end{equation}

  By symmetry, the coefficient of $a_{n-i, m-j}$ at $(i, j)$ for the horizontal recurrences
  to the left direction %%% (on $n$; those vertical recurrences are on $m$)
  is
  \begin{equation}  \label{E:left}
     \bone_{i, j}^{h} =
    (-1)^i  \sum_{i'=\max(0, i-s,j-s)}^{ \min(s, i, j)  }  \binom{s}{i'}  \binom{s}{i-i'} .
  \end{equation}
  The coefficient of $a_{n-i, m-j}$ induced by the set \POne{} is thus
  the sum of these two numbers.

  \begin{lemma} \label{L:P1}
    The coefficient of $a_{n-i, m-j}$ induced by the set \POne{} is
  \begin{equation}  \label{E:plot1}
     \bone_{i, j}^{\POnem} =
    (-1)^j  \sum_{j'=\max(0, i-s,j-s)}^{ \min(s, i, j)  }  \binom{s}{j'}  \binom{s}{j-j'}
    +
    (-1)^i  \sum_{i'=\max(0, i-s,j-s)}^{ \min(s, i, j)  }  \binom{s}{i'}  \binom{s}{i-i'} .
  \end{equation}  
  \end{lemma}

  From this expression it can be seen that for the main diagonal where $i=j$,
  the contribution of the identities of the set \POne is twice the value of Theorem~\ref{T:main}.
  \begin{coro} \label{C:P1_diag}
    The coefficient of $a_{n-i, m-i}$ along the diagonal induced by the set \POne{} is
  \begin{equation} \label{E:P1_diag}
    \bone_{i, i}^{\POnem} = (-1)^i \binom{2s}{i} \times 2,    
  \end{equation}
  which is twice the coefficient of $a_{n-i, m-j}$ of Theorem~\ref{T:main}.
  \end{coro}

\begin{figure}[!ht]
  \centering
  \begin{tikzpicture}[scale=1.2, transform shape]
  [font=\tiny]
  \tikzstyle{every node}=[font=\tiny]

  \def\s{4};
  \def\nMin{0};
  \def\mMin{0};
  \pgfmathtruncatemacro{\nMax}{(2 * \s)};
  \pgfmathtruncatemacro{\mMax}{(2 * \s)};

  \foreach \i in {\nMin, ..., \nMax} {
    \ifnum \i<\nMax
      \fill[green!20] (\i, \i) rectangle (\i+1,\i+1);
    \fi
      
    \pgfmathparse{int(mod(\i,2))}
    \ifnum 0=\pgfmathresult\relax
      \ifnum \i=\nMax
        \draw [thin] (\i,\mMin) -- (\i,\mMax)  node [below] at (\i,\mMin) {$n$};
      \else  
        \pgfmathtruncatemacro{\j}{(\nMax - \i)};
        \draw [thin] (\i,\mMin) -- (\i,\mMax)  node [below] at (\i,\mMin) {${n\!\!-\!\!\j}$};
      \fi   
    \else
      \draw [thin] (\i,\mMin) -- (\i,\mMax);
    \fi  
  }

  \pgfmathtruncatemacro{\mM}{(\nMax - 1)};
  \foreach \i in {\mMin, ..., \mM} {
    \pgfmathtruncatemacro{\j}{(\mMax - \i)};
    \draw [thin] (\nMin,\i) -- (\nMax,\i) node [right] at (\nMax,\i) {${m\!\!-\!\!\j}$};
  }
  \draw [thin] (\nMin, \mMax) -- (\nMax,\mMax) node [right] at (\nMax,\mMax) {$m$};

\begin{comment}  
  \draw [red] (7,7) node { $\downarrow$  };
  \draw [red] (6,6) node[rotate=180, yshift=0.5mm] { $\Myarrow$  };
  \draw [red] (2,2) node[rotate=180, yshift=0.5mm] { $\MyarrowS$  };
  \draw [red] (7,7) node[rotate=180] { $\Myarrow$  };
  \draw [red] (5,5) node { $A$  };
  \draw [red] (4,4) node { A  };

  \draw [ -{Stealth[red, open]} ] (0, 0);
  \draw [ -{Stealth[red]} ] (1, 2);

  \draw [red] (8,8) node[rotate=180, yshift=0.5mm] { $\Myarrow$  };
  \draw [red] (8,8) node[rotate=90, yshift=0.5mm] { $\Myarrow$  };
 
  \draw [red] (6,4) node[rotate=180, yshift=0.5mm] { $\MyarrowS$  };
  \draw [red] (6,4) node[rotate=90, yshift=0.5mm] { $\MyarrowS$  };

\end{comment}

  \foreach \i in {0, ..., \nMax} {
    \ifnum \i > \s
      \pgfmathtruncatemacro{\ii}{(\i - \s)};
      \pgfmathtruncatemacro{\jj}{(\s)};
    \else
      \pgfmathtruncatemacro{\ii}{(0)};
      \pgfmathtruncatemacro{\jj}{(\i)};
    \fi  

    \foreach \j in {\ii, ..., \jj} {
      \pgfmathtruncatemacro{\n}{(\nMax - \i)};
      \pgfmathtruncatemacro{\m}{(\mMax - \j)};

      \draw [red] (\n, \m) node[rotate=180, yshift=0.5mm] { $\Myarrow$  };
      \draw [blue] (\m, \n) node[rotate=90, yshift=0.5mm] { $\Myarrow$  };

      \pgfmathtruncatemacro{\c}{ \s! / (\j! * (\s-\j)!   };
      \pgfmathtruncatemacro{\c}{  int(mod(\j,2)) == 0 ? \c : -\c  };
      \draw [red, xshift=1.5mm, yshift=-1.5mm] (\n, \m) node {$\c$};
      \draw [blue, xshift=-1.5mm, yshift=1.5mm] (\m, \n) node {$\c$};
    }

  }

\end{tikzpicture}
  \caption{(Color online)
    The set of recurrence identities in \POne{} for $s=4$.
    The arrows labeled in red are for the vertical recurrences; those in blue
    for the horizontal recurrences.
    The numbers are the coefficients $\cob_{ij}$ (weights) for the recurrences.
    The main diagonal is highlighted in green.
    Note the original coordinates (not the abbreviated ones) are used in the figures.
  \label{F:S1}}
\end{figure}

\begin{figure}[!ht]
  \centering
  \begin{tikzpicture}[scale=1.0, transform shape]
  [font=\tiny]
  \tikzstyle{every node}=[font=\tiny]

  \def\s{5};
  \def\nMin{0};
  \def\mMin{0};
  \pgfmathtruncatemacro{\nMax}{(2 * \s)};
  \pgfmathtruncatemacro{\mMax}{(2 * \s)};

  \foreach \i in {\nMin, ..., \nMax} {
    \ifnum \i<\nMax
      \fill[green!20] (\i, \i) rectangle (\i+1,\i+1);
    \fi

    \pgfmathparse{int(mod(\i,2))}
    \ifnum 0=\pgfmathresult\relax
      \ifnum \i=\nMax
        \draw [thin] (\i,\mMin) -- (\i,\mMax)  node [below] at (\i,\mMin) {$n$};
      \else  
        \pgfmathtruncatemacro{\j}{(\nMax - \i)};
        \draw [thin] (\i,\mMin) -- (\i,\mMax)  node [below] at (\i,\mMin) {${n\!\!-\!\!\j}$};
      \fi   
    \else
      \draw [thin] (\i,\mMin) -- (\i,\mMax);
    \fi  
  }

  \pgfmathtruncatemacro{\mM}{(\nMax - 1)};
  \foreach \i in {\mMin, ..., \mM} {
    \pgfmathtruncatemacro{\j}{(\mMax - \i)};
    \draw [thin] (\nMin,\i) -- (\nMax,\i) node [right] at (\nMax,\i) {${m\!\!-\!\!\j}$};
  }
  \draw [thin] (\nMin, \mMax) -- (\nMax,\mMax) node [right] at (\nMax,\mMax) {$m$};

\begin{comment}  
  \draw [red] (7,7) node { $\downarrow$  };
  \draw [red] (6,6) node[rotate=180, yshift=0.5mm] { $\Myarrow$  };
  \draw [red] (2,2) node[rotate=180, yshift=0.5mm] { $\MyarrowS$  };
  \draw [red] (7,7) node[rotate=180] { $\Myarrow$  };
  \draw [red] (5,5) node { $A$  };
  \draw [red] (4,4) node { A  };

  \draw [ -{Stealth[red, open]} ] (0, 0);
  \draw [ -{Stealth[red]} ] (1, 2);

  \draw [red] (8,8) node[rotate=180, yshift=0.5mm] { $\Myarrow$  };
  \draw [red] (8,8) node[rotate=90, yshift=0.5mm] { $\Myarrow$  };
 
  \draw [red] (6,4) node[rotate=180, yshift=0.5mm] { $\MyarrowS$  };
  \draw [red] (6,4) node[rotate=90, yshift=0.5mm] { $\MyarrowS$  };

\end{comment}

  \foreach \i in {0, ..., \nMax} {
    \ifnum \i > \s
      \pgfmathtruncatemacro{\ii}{(\i - \s)};
      \pgfmathtruncatemacro{\jj}{(\s)};
    \else
      \pgfmathtruncatemacro{\ii}{(0)};
      \pgfmathtruncatemacro{\jj}{(\i)};
    \fi  

    \foreach \j in {\ii, ..., \jj} {
      \pgfmathtruncatemacro{\n}{(\nMax - \i)};
      \pgfmathtruncatemacro{\m}{(\mMax - \j)};

      \draw [red] (\n, \m) node[rotate=180, yshift=0.5mm] { $\Myarrow$  };
      \draw [blue] (\m, \n) node[rotate=90, yshift=0.5mm] { $\Myarrow$  };

      \pgfmathtruncatemacro{\c}{ \s! / (\j! * (\s-\j)!   };
      \pgfmathtruncatemacro{\c}{  int(mod(\j,2)) == 0 ? \c : -\c  };
      \draw [red, xshift=1.5mm, yshift=-1.5mm] (\n, \m) node {$\c$};
      \draw [blue, xshift=-1.5mm, yshift=1.5mm] (\m, \n) node {$\c$};
    }

  }

\end{tikzpicture}
  \caption{(Color online)
    The set of recurrence identities in \POne{} for $s=5$.
    Note the original coordinates (not the abbreviated ones) are used in the figures.
  \label{F:S1s5}}
\end{figure}

\subsubsection{The set of recurrence identities in \PTwo} \label{SSS:P2}

For the recurrence identities in \PTwo,
the weights used are $\coe_{i,j}$ instead of binomial coefficients used in \POne.
In the first \qu,
we put the following vertical recurrence relations at lattice site $(i, j)$:
\begin{equation*}\label{E:P2v1}
  \coe_{i,j} \e^v_{i,j}, \quad 0 \le i < s, \, i < j \le s.  
\end{equation*}
\begin{comment}
In the second \qu, at the position $(j+s, i)$,
we put
\[
  (-1)^s \coe_{s-j, s-i} \e^v_{j+s, i} , \quad 0 \le i < s, i < j \le s.  
\]
For example,
for $s=5$, 
at position $(8, 1) = (3+5, 1)$,
the coefficient is $(-1)^5 \coe_{2,4} = -125$. 
\end{comment}

In the second \qu, at the position $(i, j)$,
we put
\begin{equation*}\label{E:P2v2}
  (-1)^s \coe_{2s-i, s-j} \e^v_{i, j} , \quad s < i \le 2s,  \,  0 \le j < i-s.  
\end{equation*}

Similarly as for the set \POne,
horizontal recurrence identities are put symmetrically about the main diagonal
in the set \PTwo. Hence for \PTwo the weights are
\begin{equation} \label{E:P2}
  \begin{aligned}
    \bar{\coe}_{i,j}^v & = \begin{cases}
      \coe_{i,j}, & 0 \le i < s, \, i < j \le s \\
      (-1)^s \coe_{2s-i, s-j}, & s < i \le 2s,  \,  0 \le j < i-s
    \end{cases} \\
    \bar{\coe}_{i,j}^h & = \begin{cases}
      \coe_{j,i}, & 0 \le j < s, \, j < i \le s \\
      (-1)^s \coe_{2s-j, s-i} & s < j \le 2s,  \,  0 \le i < j-s
    \end{cases}    
  \end{aligned}  
\end{equation}

Figure~\ref{F:S2} schematically shows the recurrence 
identities in \PTwo{} for $s=4$, and Figure~\ref{F:S2s5} for $s=5$.
The numbers are the coefficients $\bar{\coe}_{i,j}$ (weights)
for the recurrence identity starting at the lattice site.

\begin{comment}
The red arrowhead at the position $(i, j)$ indicates the vertical recurrence relation identity
$\e^v_{i,j}$,
while the blue arrowhead for the horizontal recurrence identity
$\e^h_{i,j}$.
The numbers are the coefficients $\bar{\coe}_{i,j}$ (weights)
for the recurrence identity starting at the lattice site.
As the terms introduced by the recurrence identities in \POne,
terms introduced by the recurrence identities in \PTwo
are restricted inside the $(2s+1) \times (2s+1)$ square
with corners at $(0, 0)$ and $(2s, 2s)$.
\end{comment}

\begin{figure}[!ht]
  \centering
  \input{lrs_2_tikz_xint_array_input}
  \caption{(Color online)
    The set of recurrence identities in \PTwo{} for $s=4$.
    The arrows labeled in red are for the vertical recurrences; those in blue
    for the horizontal recurrences.
    The numbers are the coefficients $\bar{\coe}_{i,j}$ (weights) for the recurrences.
    Note the original coordinates (not the abbreviated ones) are used in the figures.
  \label{F:S2}}
\end{figure}

\begin{figure}[!ht]
  \centering
  \input{lrs_2_tikz_xint_array_s5_input}
  \caption{(Color online)
    The set of recurrence identities in \PTwo{} for $s=5$.
    Note the original coordinates (not the abbreviated ones) are used in the figures.
  \label{F:S2s5}}
\end{figure}

\begin{comment}
\begin{remark}
  The restrictions on $m$ and $n$ in Theorem~\ref{T:main} is $s$ greater than
  those on a lattice strip (Lemma~\ref{L:strip}).
  This is due to the fact that
  the square is of size $2s+1$.
  The conditions have to be true for the smallest lattice sites
  where the recurrence identities are added,
  which are $n-s$ and $m-s$.
  It is clear from the Figures.
\end{remark}
\end{comment}

\subsection{Toy examples for $s=4$} \label{SS:toy}
Before giving the formal proof of Theorem~\ref{T:main},
some concrete examples are given for $s=4$.
The weights for \POne (Eq.~\eqref{E:P1}) and \PTwo (Eq.~\eqref{E:P2})
are given in Figure~\ref{F:S1} and Figure~\ref{F:S2} respectively.

In this section we use the original, unabbreviated notation for the
coordinates of lattice sites.  We focus on two vertical lines
$n-5$ and $n-6$, and two horizontal lines $m-5$ and $m-6$.
These fours lines have $4$ intersections in third \qu $Q_3$,
two on the main diagonal
and two off-diagonal.
The two lattice sites on the main diagonal are
$(n-5, m-5)$ and $(n-6, m-6)$,
and the two off-diagonal lattice sites are
$(n-5, m-6)$ and $(n-5, m-6)$.

%%%%% rls_examples.mw

For the vertical line $n-5$,
the sum of the strip recurrences of Eq.~\eqref{E:rec1} from \POne
(using the weights from Eq.~\eqref{E:P1}, as shown in Figure~\ref{F:S1})
is given by %%% P1v_5
\begin{equation} \label{E:P1v_5}
\begin{aligned}
  S_{n-5}^{\POnem-v} &=
    -4  \sum_{j=0}^4  (-1)^j  \binom{4}{j}  a_{n-5, m-1-j} 
    +6  \sum_{j=0}^4  (-1)^j  \binom{4}{j}  a_{n-5, m-2-j}  \\
  &\phantom{{}=}
    -4  \sum_{j=0}^4  (-1)^j  \binom{4}{j}  a_{n-5, m-3-j} 
    +1  \sum_{j=0}^4  (-1)^j  \binom{4}{j}  a_{n-5, m-4-j} ,
  \end{aligned}
\end{equation}  
and
the sum from \PTwo
(using the weights from Eq.~\eqref{E:P2}, as shown in Figure~\ref{F:S2})
is given by  %%% P2v_5 
\begin{equation} \label{E:P2v_5}
   S_{n-5}^{\PTwom-v} = -13 \sum_{j=0}^4 (-1)^j \binom{4}{j}  a_{n-5, m-j} .
\end{equation}

Similarly, for the vertical line $n-6$,  %%% P1v_6
\begin{equation} \label{E:P1v_6}
\begin{aligned}
  S_{n-6}^{\POnem-v} &=
     6  \sum_{j=0}^4  (-1)^j  \binom{4}{j}  a_{n-6, m-2-j} 
    -4  \sum_{j=0}^4  (-1)^j  \binom{4}{j}  a_{n-6, m-3-j}  \\
  &\phantom{{}=}
    +1  \sum_{j=0}^4  (-1)^j  \binom{4}{j}  a_{n-6, m-4-j},
  \end{aligned}
\end{equation}    
and  %%% P2v_6
\begin{equation} \label{E:P2v_6}
  S_{n-6}^{\PTwom-v} =
    15 \sum_{j=0}^4 (-1)^j \binom{4}{j}  a_{n-6, m-j}
  + 24 \sum_{j=0}^4 (-1)^j \binom{4}{j}  a_{n-6, m-1-j}.
\end{equation}

For the horizontal line $m-5$, we have  %%% P1h_5
\begin{equation} \label{E:P1h_5}
\begin{aligned}
  S_{m-5}^{\POnem-h} &=
    -4  \sum_{i=0}^4  (-1)^i  \binom{4}{i}  a_{n-1-i, m-5} 
    +6  \sum_{i=0}^4  (-1)^i  \binom{4}{i}  a_{n-2-i, m-5}  \\
    &\phantom{{}=}
    -4  \sum_{i=0}^4  (-1)^i  \binom{4}{i}  a_{n-3-i, m-5} 
    +1  \sum_{i=0}^4  (-1)^i  \binom{4}{i}  a_{n-4-i, m-5} ,
  \end{aligned}
\end{equation}  
and %%% P2h_5
\begin{equation} \label{E:P2h_5}
   S_{m-5}^{\PTwom-h} = -13 \sum_{i=0}^4 (-1)^i \binom{4}{i}  a_{n-i, m-5} .
\end{equation}

Similarly, for the horizontal line $m-6$,  %%% P1h_6
\begin{equation} \label{E:P1h_6}
\begin{aligned}
  S_{m-6}^{\POnem-h} &=
     6  \sum_{i=0}^4  (-1)^i  \binom{4}{i}  a_{n-2-i, m-6} 
    -4  \sum_{i=0}^4  (-1)^i  \binom{4}{i}  a_{n-3-i, m-6}  \\
    &\phantom{{}=}
    +1  \sum_{i=0}^4  (-1)^i  \binom{4}{i}  a_{n-4-i, m-6},
  \end{aligned}
\end{equation}    
and  %%% P2h_6
\begin{equation} \label{E:P2h_6}
  S_{m-6}^{\PTwom-h} =
    15 \sum_{i=0}^4 (-1)^i \binom{4}{i}  a_{n-i, m-6}
  + 24 \sum_{i=0}^4 (-1)^i \binom{4}{i}  a_{n-1-i, m-6}.
\end{equation}

\subsubsection{Diagonal terms $(n-5, m-5)$ and $(n-6, m-6)$}

%%% P1v_5 + P1h_5 + P2v_5 + P2h_5
For the diagonal terms $(n-5, m-5)$,
the only contributions are from recurrence terms along the vertical line $n-5$
and the horizontal line $m-5$.  Direct calculation shows that the coefficient of $a_{n-5, m-5}$
  of the four sums (Eqs.~\eqref{E:P1v_5}, \eqref{E:P2v_5}, \eqref{E:P1h_5}, and \eqref{E:P2h_5})
  is
\[
  [a_{n-5, m-5}] \left(
    S_{n-5}^{\POnem-v} + S_{n-5}^{\PTwom-v} + S_{m-5}^{\POnem-h} + S_{m-5}^{\PTwom-h}
  \right)
  = -56 + 0  - 56 + 0
  = -112 = (-1)^5 \binom{2 \times 4}{ 5} \times 2,
\]
which agrees with Eq.~\eqref{E:2s} of Theorem~\ref{T:main}.

%%% P1v_6 + P1h_6 + P2v_6 + P2h_6;
Similarly for the diagonal terms $(n-6, m-6)$,
the only contributions are from recurrence terms along the vertical line $n-6$
and the horizontal line $m-6$.
Direct calculation shows that the coefficient of $a_{n-6, m-6}$
  of the four sums (Eqs.~\eqref{E:P1v_6}, \eqref{E:P2v_6}, \eqref{E:P1h_6}, and \eqref{E:P2h_6})
  is
\[
  [a_{n-6, m-6}] \left(
    S_{n-6}^{\POnem-v} + S_{n-6}^{\PTwom-v} + S_{m-6}^{\POnem-h} + S_{m-6}^{\PTwom-h}
  \right)
  = 28+0+28+0
  = 56 = (-1)^6 \binom{2 \times 4}{ 6} \times 2,
\]
which agrees with Eq.~\eqref{E:2s} of Theorem~\ref{T:main}.

\subsubsection{Off-diagonal terms $(n-5, m-6)$ and $(n-6, m-5)$}

%%% P1v_5 + P1h_6 + P2v_5 + P2h_6;
For the off-diagonal terms $(n-5, m-6)$,
the only contributions are from recurrence terms along the vertical line $n-5$
and the horizontal line $m-6$.  Direct calculation shows that the coefficient of $a_{n-5, m-6}$
  of the four sums (Eqs.~\eqref{E:P1v_5}, \eqref{E:P2v_5}, \eqref{E:P1h_6}, and \eqref{E:P2h_6})
  is
\[
  [a_{n-5, m-6}] \left(
    S_{n-5}^{\POnem-v} + S_{n-5}^{\PTwom-v} + S_{m-6}^{\POnem-h} + S_{m-6}^{\PTwom-h}
  \right)
  = 28+0-52+24
  = 0.
\]

%%% P1v_6 + P1h_5 + P2v_6 + P2h_5;
Similarly for the off-diagonal terms $(n-6, m-5)$,
the only contributions are from recurrence terms along the vertical line $n-6$
and the horizontal line $m-5$.  Direct calculation shows that the coefficient of $a_{n-6, m-5}$
  of the four sums (Eqs.~\eqref{E:P1v_6}, \eqref{E:P2v_6}, \eqref{E:P1h_5}, and \eqref{E:P2h_5})
  is
\[
  [a_{n-6, m-5}] \left(
    S_{n-6}^{\POnem-v} + S_{n-6}^{\PTwom-v} + S_{m-5}^{\POnem-h} + S_{m-5}^{\PTwom-h}
  \right)
  = -52+24+28+0
  = 0.
\]

\section{Proof of the main theorem} \label{S:proof}

\subsection{The first \qu{} $Q_1$} \label{SS:Q1}  

From Figure~\ref{F:S2} and Figure~\ref{F:S2s5} we see that the recurrence identities in \PTwo
do not introduce terms in the main diagonal inside the first \qu{} $Q_1$ (the top right \qu{}).
Let consider the strict lower-right triangle of the first \qu,
which excludes the main diagonal.
In this triangle,
both horizontal and vertical recurrence identities in \POne
introduce terms (Figure~\ref{F:S1} and Figure~\ref{F:S1s5}),
while only 
vertical recurrences in \PTwo
have effects (Figure~\ref{F:S2} and Figure~\ref{F:S2s5}). 
These identities are $\e^v_{i,j}$,
where 
$0 \le i < s$ and $i+1 \le j \le s$.
Below we prove Eq.~\eqref{E:goal_lhs}
for the first \qu.

\begin{proof}[Partial proof of Theorem~\ref{T:main} in the first \qu{}]
  For the lattice site $(i, j)$ of the strict lower-right triangle of the first \qu{} $Q_1$,
  $0 \le i < s$ and $i+1 \le j \le s$.
  The coefficient of $a_{n-i, m-j}$ contributed by the identities from \POne
  is given by Eq.~\eqref{E:plot1} of Lemma~\ref{L:P1},
  which simplifies to
  \begin{align} \label{E:Q1_S1}
    \bone_{i,j}
    &=
      (-1)^j \sum_{j'=0}^{i} \binom{s}{j'} \binom{s}{j-j'} 
      + (-1)^i \sum_{i'=0}^{i}  \binom{s}{i'} \binom{s}{i-i'}   \nonumber \\
    &= (-1)^j \sum_{j'=0}^{i} \binom{s}{j'} \binom{s}{j-j'} + (-1)^i \binom{2s}{i} ,
  \end{align}
where the binomial identity Eq.~\eqref{E:binom_iden_V}
is used in the second step.

The contribution from identities of \PTwo is given by
\begin{equation*}
  \btwo_{i,j}
  = \sum_{j'=i+1}^{j} (-1)^{j-j'} \coe_{i, j'} \binom{s}{j-j'} 
  = \sum_{j'=0}^{j} (-1)^{j-j'} \coe_{i, j'} \binom{s}{j-j'},
\end{equation*}
where we use the fact that $\coe_{i, j} = 0$ for $j \le i$ to extend the range of summation
in the second step.
If we did not change the range of the summation, the proof would be more involved
for the third term of $\coe_{i, j}$, due to nonstandard boundary conditions.

The contribution from the first term of $\coe_{i,j}$ is
\begin{equation} \label{E:Q1_h1}
   \btwo_{i,j}^{(1)} = (-1)^{j+1} \sum_{j'=0}^{i} \binom{s}{j'} \binom{s}{j-j'}  .
\end{equation}

The contribution from the second term of $\coe_{i,j}$ is
  \begin{align} \label{E:Q1_h2}
    \btwo_{i,j}^{(2)} &=
   \sum_{j'=i+1}^{j} (-1)^{j-j'} \coe_{i, j'}^{(2)} \binom{s}{j-j'}  \nonumber \\
                      &=  (-1)^{i+1}
                        \binom{2s}{i}
                        \sum_{j'=i+1}^j (-1)^{j-j'}
                        \binom{s+j'-i-1}{ s}
                        \binom{s}{j-j'} \nonumber  \\
    &= (-1)^{i+1}  \binom{2s}{i} .
  \end{align}
  The summation in the second step is evaluated to $1$,
  which can be obtained by Lemma~\ref{L:binom_iden} in Appendix~\ref{S:A_binom},
  by making variable changes $j'-i-1 \to j'$ and $j-i-1 \to j$.
  The identity is a special case of $t=0$ in Eq.~\eqref{E:binom_iden}.
  
The contribution from the third term of $\coe_{i,j}$ is,
\[
   \btwo_{i,j}^{(3)} =  \sum_{j'=0}^{j} (-1)^{j-j'} \coe_{i, j'}^{(3)} \binom{s}{j-j'} .
 \]
Exchange the order of summation in the above equation,
the inner sum becomes
\[
  \sum_{j'=0}^j
  (-1)^{j'}
  \binom{s-t-1+j'} { j'}
  \binom{s} {j-j'} .
\]
Lemma~\ref{L:binom_iden} shows that the sum vanishes.

Combining above results of Eq.~\eqref{E:Q1_S1}, Eq.~\eqref{E:Q1_h1},
and Eq.~\eqref{E:Q1_h2},
we obtain
\[
  \bone_{i,j} + \btwo_{i,j} = 0
\]
in the strict lower-right triangle of the first \qu{} $Q_1$.

By symmetry, we can also prove a similar result
for the strict upper-left triangle of the first \qu{} $Q_1$:
the contributions of identities from \POne and \PTwo sum to zero.
Since the recurrences in \PTwo do not affect the main diagonal
in the first \qu,
together with Corollary~\ref{C:P1_diag},
we prove that
Eq.~\eqref{E:goal_lhs} holds in the first \qu $Q_1$.
\end{proof}

\subsection{The third \qu $Q_3$}

\subsubsection{Summary}

The third \qu (bottom left \qu{}) is more complicated than the first \qu.
In this section we focus on the strict upper-left triangle of the third \qu,
which excludes the main diagonal.
Once we prove the theorem in this triangle,
by symmetry we also prove the theorem
in the strict lower-right triangle.
The case of diagonal in the third \qu has its unique characteristics,
and is discussed in section~\ref{SS:Q3_diag}.

On the strict upper-left triangle above the
diagonal
of the third \qu where $i>j$,
only the vertical recurrences from \PTwo (those arrowheads colored in red of Figure~\ref{F:S2}
and Figure~\ref{F:S2s5})
contribute,
while both vertical and horizontal recurrences from \POne contribute in this triangle
(Figure~\ref{F:S1} and Figure~\ref{F:S1s5}).

We prove that the contributions from the first term of $\bar{\coe}_{i,j}$ in \PTwo cancel
those from the vertical recurrences of \POne,  %%% sp11
\[
  \bone_{i,j}^{v} + \btwo_{i,j}^{(1)} = 0,
\]
the contributions from the second term of $\bar{\coe}_{i,j}$ in \PTwo cancel
those from the horizontal recurrences of \POne, %%% sp12
\[
  \bone_{i,j}^{h} + \btwo_{i,j}^{(2)} = 0,
\]
while the contributions from the third term of $\bar{\coe}_{i,j}$ in \PTwo vanish
\[
  \btwo_{i,j}^{(3)} = 0 .
\]
To prove $\btwo_{i,j}^{(3)} = 0$,
we transform the double sum involving the third term of $\coe_{i,j}$
into a new form as a sum of a simple term and a new double sum.
The motivation for this transform is that
the new double sum satisfies a simple recurrence relation. 
With this new form,
we show that these two terms cancel each other.

\begin{comment}
  \begin{itemize}
\item
  The first term of $\coe_{2s-i, s-j'}$ in the sum of $ a_{n-i, m-j}^{(S_{2v})} $ 
  cancels the first term (sp1) in Eq.~\eqref{E:Q3_S1}, similarly as for Q4;
  This is proved below;
\item
  The second term of $\coe_{2s-i, s-j'}$ cancels the second term (sp2) in Eq.~\eqref{E:Q3_S1};
  This is proved below in Lemma~\ref{L:Q3_S20v} and Lemma~\ref{L:Q3_sp2};
\item
  The third term of $\coe_{2s-i, s-j'}$ is zero:
  the single term cancels the double sum.
  We need to prove
  that the double sum is
  \[
    (-1)^{j+1} \binom{2s} {j} .
  \]
\end{itemize}
\end{comment}

In the third \qu at position $(i, j)$, the vertical recurrences in \PTwo contribute
(see section~\ref{SSS:P2},  Figure~\ref{F:S2} and Figure~\ref{F:S2s5})
\begin{equation} \label{E:Q3_v}
  \btwo_{i,j}^{v} = (-1)^s \sum_{j'=0}^{s} (-1)^{j-j'} \coe_{2s-i, s-j'} \binom{s}{j-j'} .
\end{equation}
We deal each of the three terms of $\coe_{i,j}$ separately below.

\subsubsection{ The contributions of  $\coe_{2s-i, s-j}^{(1)}$
  cancels the vertical recurrences of \POne } \label{SSS:Q3_s1g1}

The first term of $\coe_{i, j}$ is
$
  \coe_{i, j}^{(1)} = (-1)^{j+1} \binom{s}{j} [j \le i] .
$
This term in the sum of $\btwo_{i,j}^{v}$ above (Eq.~\eqref{E:Q3_v}) is simplified as
\begin{align} \label{E:Q3_s1g1}
  \btwo_{i,j}^{(1)} &= 
   (-1)^s \sum_{j'=0}^{s} (-1)^{j-j'} (-1)^{s-j'+1} \binom{s}{s-j'} \binom{s}{j-j'} [s-j' \le 2s-i] \nonumber \\
  &= (-1)^{j+1} \sum_{j'=i-s}^{s}  \binom{s}{j'} \binom{s}{j-j'},
\end{align}
which cancels the contributions from
the vertical recurrences of \POne (Eq.~\eqref{E:down}).

\begin{lemma} \label{L:Q3_P1vP2_1}
  \[
    \bone_{i,j}^{v} + \btwo_{i,j}^{(1)} = 0 .
  \]
\end{lemma}

\begin{comment}
The first term of $\coe_{2s-j, s-i'}$ in the sum of $ a_{n-i, m-j}^{(S_{2h})} $ 
cancels the second term in Eq.~\eqref{E:Q3_S1}, similarly as for Q4.
\end{comment}

\subsubsection{The contributions of $\coe_{2s-i,s-j}^{(2)}$
 cancel the horizontal recurrences of \POne} \label{ss:Q3_s1g2}

In the upper-left triangle of the third \qu $Q_3$,
the contributions from  the horizontal recurrences of \POne are
\[
  (-1)^i  \sum_{i'=j-s}^s \binom{s}{i'}  \binom{s}{i-i'} .
\]
Since $i \ge j$,
the boundary conditions in the sum are standard, and Chu-Vandermonde identity
Eq.~\eqref{E:binom_iden_V} applies.
\begin{lemma} \label{L:Q3_sp2}
  For $i \ge j$ in upper-left triangle,
  the contributions from the horizontal recurrences of \POne
  \begin{equation} \label{E:Q3_sp2}
    \bone_{i,j}^{h} =
    (-1)^i  \sum_{i'=j-s}^s \binom{s}{i'}  \binom{s}{i-i'} = (-1)^{i} \binom{2s}{i} .    
  \end{equation}
\end{lemma}

The sum of the second term of $\coe_{2s-i, s-j'}$ in the upper-left triangle of the third \qu $Q_3$ takes
the form of
\[
  \btwo_{i,j}^{(2)} =  (-1)^{s+i+j+1}  \binom{2s}{i} \sum_{j'=0}^{i-1}  (-1)^{j'}  \binom{i-j'-1} {s}  \binom{s} {j-j'} .
\]

Let first prove the following. 
\begin{lemma} \label{L:Q3_S10}
  \begin{equation} \label{E:Q3_S10c}
    \sum_{j'=0}^{i-1}  (-1)^{j'}  \binom{i-j'-1} {s}  \binom{s} {j-j'} =
  \begin{dcases}
    (-1)^{s+j} ,  & \quad s \le j < i, \\
    0, & \text{otherwise} .
    \end{dcases}    
  \end{equation}
  \end{lemma}
  
\begin{proof}
  Denote the sum on the left as $S(j)$ and the summand as $\F(j, j')$.
  \Ze's algorithm shows that the summand satisfies the first order homogeneous recurrence
  \begin{equation} \label{E:Q3_S20_recur}
    (i-j-1) \F(j, j') + (i-j-1) \F(j+1, j') = 0, 
  \end{equation}
  with certificate
  \[
    \R = \frac{ (i-j') (-s+j-j') } {j+1-j' } .
  \]
  Note that the %highest order
  coefficient polynomial of  $F(j+1, j')$ in the recurrence
  has a zero at $j=i-1$,
  so the recurrence breaks up
  at the term $\Sum(i)$.

  For $j \le i-1$,
  let evaluate $\Sum(i-1)$.
  In this case the summand becomes
  \[
     (-1)^{j'}  \binom{i-j'-1} {s}  \binom{s} {i-j'-1},
   \]
   which vanishes except when $j'=i-1-s$,
   hence the sum is evaluated as $S(i-1) = (-1)^{s+i+1}$.
   From the recurrence Eq.~\eqref{E:Q3_S20_recur}
   $S(j) = (-1)^{s+j}$ when $j<i$.
   
   When $j \ge i$,
   the condition for the first binomial in the summand to be nonzero
   is
   \[
     j' \le i-s-1,
   \]
   while 
   the condition for the second binomial to be nonzero is
   \[
     j' \ge j -s \ge i - s,
   \]
   which show that the summand is zero for all $j'$,
   hence $S(j) = 0$ for $j \ge i$.
   
\end{proof}

From this identity we have
\begin{coro}
  In the upper-left triangle of the third \qu $Q_3$,
  the contribution of the second term of $\coe_{i,j}$ from \PTwo is
  \begin{equation} \label{E:Q3_S10}
  \btwo_{i,j}^{(2)} =
    (-1)^{i+1} \binom{2s}{i} .
    \end{equation}
  \end{coro}

\begin{comment}
\begin{equation} 
  (-1)^{s+i+j+1}  \binom{2s}{i} \sum_{j'=0}^{i-1}  (-1)^{j'}  \binom{i-j'-1} {s}  \binom{s} {j-j'} =
  \begin{cases}
    (-1)^{i+1} \binom{2s}{i},  & \quad s \le j < i, \\
    0, & \text{otherwise} .
    \end{cases}
\end{equation}
\end{comment}

By Eq.~\eqref{E:Q3_sp2} and Eq.~\eqref{E:Q3_S10} we prove that
\begin{coro} \label{C:Q3_P1hP2_2}
  In the upper-left triangle of the third \qu $Q_3$,
  \[
    \bone_{i,j}^{h} + \btwo_{i,j}^{(2)} = 0.
  \]
\end{coro}

\subsubsection{The contributions of $\coe_{2s-i, s-j'}^{(3)}$ vanish
} \label{ss:Q3_h2}

\paragraph{Outline of the proof}
In this section we prove $\btwo_{i,j}^{3v} = 0$ in the strict upper-left triangle of the third \qu.
Compared to the proofs of contributions from $\coe_{2s-i, s-j'}^{(1)}$ and
$\coe_{2s-i, s-j'}^{(2)}$, this proof is more complicated, since a double sum is involved.
The third term $\coe_{2s-i, s-j'}^{(3)}$ itself is a sum,
thus by Eq.~\eqref{E:Q3_v} the contributions from the third term in upper-left triangle where $i>j$
are represented as a double sum (Eq.~\eqref{E:Q3_double_sum} below).
To prove $\btwo_{i,j}^{3v} = 0$,
we transform the double sum into two parts,
the first part is a simple closed form expression,
and the second part is a new double sum that satisfies a simple recurrence.
We then prove that these two parts cancel each other.

To do this,
we transform the inner sum of the double sum
\[
  \btwo_{i,j}^{3v} = (-1)^s \sum_{j'=0}^{s} (-1)^{j-j'} \coe_{2s-i, s-j'}^{(3)} \binom{s}{j-j'}
\]
to two parts (Proposition~\ref{P:Q3_sum_t}).
The first part is a simple closed form expression (Lemma~\ref{L:Q3_sum_t_closed_form}),
and the second part is a new sum.
This new sum is a correction term that is needed when $j' < i-s$ (Lemma~\ref{L:Q3_correction}).
This is due to a zero in the coefficient polynomial of the lowest recurrence term,
which interrupts the recurrence in the downwards direction of recurrence on $j'$.
Some of the results of this section will be used in the section
for the fourth \qu $Q_4$.

\paragraph{Transforming the inner sum}

The explicit form of full double sum involving $\coe_{2s-i, s-j'}^{(3)}$
of $\btwo_{i,j}^v$ is
\begin{equation} \label{E:Q3_double_sum}
  (-1)^{s+i} i  \binom{2s}{ i}
  \sum_{j'=0}^{s} (-1)^{j-j'} \binom{s} {j-j'}
  \sum_{t=0}^{2s-i}  \frac{(-1)^t }{ 2s - t}   \binom{2s-i}{ t} \binom{2s-t-1-j'} { s-j'} .
\end{equation}

The explicit form of the inner sum with $t$ as the summation index is
\begin{equation} \label{E:Q3_sum_t}
  \sum_{t=0}^{2s-i}  \frac{(-1)^t }{ 2s - t}   \binom{2s-i}{ t} \binom{2s-t-1-j'} { s-j'} .
\end{equation}

\Ze's algorithm shows the summand obeys a second order recurrence on $j'$
\begin{multline} \label{E:Q3_sum_t_recur}
  -(j'-s)(i-j'-s-1) \F(j',t)
  + (2ij'-i s-2 j'^2+2i-5j'+s-3) \F(j'+1,t) \\
  -(j'+2)(i-j'-2) \F(j'+2,t) = 0,
\end{multline}
with the certificate
\[
  R =  \frac{ t(j'-s)(2s-t) } {(-2s+t+1+j') } .
\]

Note that the lowest coefficient polynomial of the recurrence
(the coefficient of $\F(j',t)$)
has a zero at $j'=i-s-1$, which is within the range of summation in the outer sum of variable $j'$.

Below we show that when $j' \ge i-s$,
the inner sum has a closed form.

\begin{lemma} \label{L:Q3_sum_t_closed_form}
  When $j' \ge i-s$,
  the sum of Eq.~\eqref{E:Q3_sum_t} has the following closed form:
  \begin{equation} \label{E:Q3_sum_t_closed_form}
    \sum_{t=0}^{2s-i}  \frac{(-1)^t }{ 2s - t}   \binom{2s-i}{ t} \binom{2s-t-1-j'} { s-j'}
    =
    \frac{ (-1)^{s+i+j'}}{i}
    { \dbinom{s}{ j'} } { \dbinom{2s} {i} }^{-1}.
  \end{equation}
\end{lemma}

\begin{proof}
  It is easy to check that
  the closed form on the right side of Eq.\eqref{E:Q3_sum_t_closed_form}
  obeys the same second order recurrence Eq.~\eqref{E:Q3_sum_t_recur}.
  We will show that both sides have the same values
  at $j'=s$ and $j'=s-1$.

  When $j'=s$,
  the summand becomes
  \[
     \F(i, t) = \frac{(-1)^t }{ 2s - t} \binom{2s-i}{ t}.
   \]
   The summand obeys the first order homogeneous recurrence on $i$
   \[
     i \F(i, t) + (-i+2s) \F(i+1, t) = 0,
   \]
   with the certificate
   \[
    \R = \frac{(2s-t)t}{i-2s} .
  \]
  The sum can be evaluated from the recurrence as
  \[
    \frac{(-1)^i } {i}
    { \dbinom{2s}{ i}}^{-1},
  \]
  which equals the right hand side of Eq.~\eqref{E:Q3_sum_t_closed_form} when $j'=s$.

  For $j'=s-1$,
  the summand is
  \[
    \F(i, t) = (-1)^t  \frac{s-t}{ 2s - t} \binom{2s-i}{ t}.
  \]
  which obeys the first order homogeneous recurrence
  \[
    -(i-2s+1) F(i, t) + (i-2s+1)(i-2s) \F(i+1, t) = 0, 
  \]
  with the certificate
  \[
    \R = \frac{ -(2s-t)t(i s-2s^2-s+2st-i t+i) } {(s-t)(i-2s) }.
  \]
  The recurrence is solved to yield
  \[
    \frac{(-1)^{i+1}  s} { i }
    \dbinom{2s}{ i}^{-1},
  \]
  which agrees with the right hand side of Eq.~\eqref{E:Q3_sum_t_closed_form} when $j'=s-1$,
  hence the identity of Eq.~\eqref{E:Q3_sum_t} is proved.
\end{proof}

The above Lemma only holds when $j' \ge i-s$.
For other values of $j'$ that we need in the sum of Eq.~\eqref{E:Q3_double_sum},
that is, $0 \le j' \le i-s -1$,
a correction term is needed.
This is due to a zero in the lowest coefficient polynomial of
the recurrence Eq.~\eqref{E:Q3_sum_t_recur}.
The coefficient polynomial of $\F(j',t)$ vanishes at
\[
  j'=i-s-1.
\]
Because of this zero, the recurrence cannot go on without ambiguity
on the downshift direction of $j'$.

The correction term is given in the following Lemma.
\begin{lemma}[The correction term] \label{L:Q3_correction}
  When $j' < i-s$, the correction term is
  \begin{equation}  \label{E:Q3_correction}
    \sum_{t=1}^{s-j'}
     \frac{ (-1)^{t+1}}{t}
     \dbinom{i-j'-1}{ s+t-1}  \dbinom{s}{ t-1}   {\dbinom{s-j'} {t} }^{-1}.    
  \end{equation}
\end{lemma}

To prove the Lemma, we need first to show that
the value of the correction term when $j'=i-s-1$,
plus the value of the right side of Eq.~\eqref{E:Q3_sum_t_closed_form}
evaluated at  $j'=i-s-1$,
is the correct value of the original sum (Eq.\eqref{E:Q3_sum_t}).
This is done in Lemma~\ref{L:Q3_breakpoint}.

 \begin{lemma}[Value at $j'=i-s-1$] \label{L:Q3_breakpoint}
   When $j'=i-s-1$,
   the original sum of Eq.~\eqref{E:Q3_sum_t} equals to
   \begin{equation} \label{E:Q3_corr_is1}
     \frac{1} {2s -i +1} +
     \frac{(-1)^{2i-1} } {i}
     \dbinom{s}{ -s+i-1}  { \dbinom{2s} {i} }^{-1}.
   \end{equation}
   where the first term is the value of the correction term Eq.~\eqref{E:Q3_correction} at $j'=i-s-1$,
   and the second term is the value of the right side of Eq.~\eqref{E:Q3_sum_t_closed_form}
   at $j'=i-s-1$.
 \end{lemma}

 \begin{proof}

   When $j'=i-s-1$,
   the  summand of the original sum of Eq.~\eqref{E:Q3_sum_t} obeys a homogeneous second order recurrence
     \[
       (i-2s-1)i \F(i, t)  -(2 i-s+1) (i-2s) \F(i+1, t) + (i-s+1)(i-2s+1)  \F(i+2, t) = 0,
     \]
      with the certificate
      \[
        -\frac{ (i-2s-1)s(2 s-t)t } {(i-2s)(-3s+t+i) } .
      \]

      It can be shown that expression in Eq.~\eqref{E:Q3_corr_is1} obeys this recurrence
      by direct calculation.

      For initial conditions,
      when $i=2s$,
      the first binomial coefficient of the summand of the original sum vanishes for all values of
      $t$ except when $t=0$.
      In this case, the sum takes a value of $1/2$,
      which agrees with the expression Eq.~\eqref{E:Q3_corr_is1}.

      when $i=2s-1$,
      the first binomial coefficient of the summand of the original sum vanishes for all values of
      $t$ except when $t=0$ or $t=1$.
      The sum of these two terms are
      \[
        \frac{1}{4}
        \frac{3s-1} {2s-1} ,
      \]
      which agrees with the expression Eq.~\eqref{E:Q3_corr_is1}
      with direct substitutions.
      This proves the Lemma.
   \end{proof}

   Next we prove that the correction term Eq.\eqref{E:Q3_correction} vanishes
   when $j' \ge i-s$.
 \begin{lemma}[The correction term vanishes when $j' \ge i-s$]  \label{L:Q3_correction_zero}
   When $j' \ge i-s$,
   the correction term Eq.\eqref{E:Q3_correction} vanishes.
 \end{lemma} 
 
 \begin{proof}
   The first binomial coefficient in the numerator of the summand of
   Eq.\eqref{E:Q3_correction}
   vanishes when $j' > i-s-t$. Since the index of summation $t \ge 1$,
   this binomial coefficient vanishes for all the range of $t$, so is the sum.
 \end{proof}

 By the the previous two Lemmas,
 we proceed to prove the correctness of Lemma~\ref{L:Q3_correction}.
\begin{proof}[Proof of Lemma~\ref{L:Q3_correction}]
  When $j'=i-s-1$,
  the summand of Eq.\eqref{E:Q3_correction} is nonvanishing only when $t=1$,
  and
  the above correction term takes a value of
  \[
    \frac{1}{2s-i+1} .
  \]

  The summand of the correction term
  \[
    \F(j', t) =  \frac{ (-1)^{t+1}}{t}
      \dbinom{i-j'-1}{ s+t-1}  \dbinom{s}{ t-1}   {\dbinom{s-j'} {t} }^{-1}
   \]
   satisfies the first order inhomogeneous recurrence (with $j'$ as the parameter)
   \begin{equation} \label{E:Q3_corr_recur}
     (-j'+s)\F(j', t) + (j'+1) \F(j'+1, t) =
    \frac{ s(j'-s)} {(-j'+s)(i-j'-1) }   \binom{i-j'-1}{ s },
   \end{equation}
   with the certificate
   \[
     \R = \frac{ (s+t-1)(j'-s) } {i-j'-1 }.
   \]

   It can be checked that the inhomogeneous recurrence
   of Eq.~\eqref{E:Q3_corr_recur} obeys the 
   recurrence Eq.~\eqref{E:Q3_sum_t_recur} that original sum satisfies.

   The only things left are to confirm
   that the initial conditions also agree.
   In Lemma~\ref{L:Q3_correction_zero}
   we show the correction term vanishes when $j'>=i-s$,
   and in Lemma~\ref{L:Q3_breakpoint}
   we show the initial condition agrees at $j'=i-s-1$.
   By taken these results together,
   the correctness of Eq.\eqref{E:Q3_correction}
   is proved.

 \end{proof}

   Putting all above results together, we prove in the upper-left triangle of the third \qu $Q_3$,
   for $0 \le j' \le s$, the following result holds.
   \begin{prop}[Transform of the inner sum] \label{P:Q3_sum_t}
     For $0 \le j' \le s$,
     the inner sum of $\btwo_{i,j}^{3v}$ is transformed to
\begin{align}
 & \sum_{t=0}^{2s-i}  \frac{(-1)^t }{ 2s - t}   \binom{2s-i}{ t} \binom{2s-t-1-j'} { s-j'} \nonumber \\
  =&
   \frac{ (-1)^{s+i+j'}}{i}
     \dbinom{s}{ j'}  { \dbinom{2s} {i} }^{-1}
     +
     \sum_{t=1}^{s-j'}
     \frac{ (-1)^{t+1}}{t}
     \dbinom{i-j'-1}{ s+t-1}  \dbinom{s}{ t-1}   {\dbinom{s-j'} {t} }^{-1}
     \label{E:Q3_sum_t_corr}.
\end{align}
\end{prop}

Denote the first term of the rhs of
Eq.~\eqref{E:Q3_sum_t_corr} by
$\sumtS(s, i, j')$.
By putting the outer summation back (Eq.~\eqref{E:Q3_double_sum}),
and by using identity Eq.~\eqref{E:binom_iden_V},
the sum of $\sumtS(s, i, j')$ can be evaluated to a closed form.
\begin{lemma} \label{L:Q3_inner_sum_1st_term}
  \[
    (-1)^{s+i+j}  i \binom{2s}{i}
      \sum_{j'=0}^{s}  (-1)^{j'}  \sumtS(s, i, j')  \binom{s}{j-j'}
    =
    (-1)^j \binom{2s} {j} .
  \]
\end{lemma}

Hence in the strict upper-left triangle of the third \qu $Q_3$ we transform
the double sum involving the third term of
$\coe_{i,j}$ into another form.
\begin{prop}[Transformed double sum] \label{P:Q3_S1X} %%% same label???
  In the strict upper-left triangle of the third \qu $Q_3$, the double summation
  of the third term of
  $\btwo^v_{i,j}$
  takes a new form of
  \begin{equation} \label{E:Q3_S1X}
    \btwo_{i,j}^{3v} =
    (-1)^j \binom{2s}{j}
    +
    \D(s, i, j),    
  \end{equation}
  with
  \begin{equation} \label{E:Q3_double_sum_new}
    \D(s, i, j) =
    (-1)^{s+i+j} (i) \binom{2s} {i}
    \sum_{j'=0}^{s}
    \sum_{t=1}^{s-j'} \frac{(-1)^{j'+t+1}}{t}
    \binom{i-j'-1}{s+t-1}
    \binom{s} {t-1}
    \binom{s}{ j-j'}
    { \dbinom{s-j'}{ t} }^{-1}.
  \end{equation}
\end{prop}

\paragraph{Recurrence relation for the new double sum $\D(s, i, j)$}

With this new transformed form of $\btwo_{i,j}^{3v}$,
we can obtain a simple recurrence for the double sum
using the multivariable version of \Ze's algorithm~\cite{wilfAlgorithmicProofTheory1992,Apagodu2006}.
\begin{prop}[Recurrence of the double sum] \label{P:Q3_double_sum_recur}
  The double sum satisfies the following first order inhomogeneous recurrence
  \begin{equation}  \label{E:Q3_double_sum_recur}
    -\D(s, i, j) + \D(s, i+1, j) = \Inhom(s, i, j),
  \end{equation}
  where the inhomogeneous term $\Inhom(s, i, j)$
  is described in Proposition~\ref{P:Q3_inhom}.
\end{prop}

\begin{proof}
  If we leave out the factor $(-1)^{s+j+1}$, the summand
  \[
    \F(i; j', t) =
    (-1)^{i+t+j'}
    \frac{i}{t}
    \binom{2s}{ i}
    \binom{(i-j'-1}{ s+t-1}
    \binom{s}{ t-1}
    \binom{s} {j-j'} 
    { \dbinom{s-j'}{ t} }^{-1}
  \]
  satisfies a first order recurrence
  \[
    -F(i; j', t) + F(i+1; j', t) = \Delta_{j'}(i; j', t) + \Delta_{t}(i; j', t),
  \]
  where $\Delta_{j'}$
  and
  $\Delta_{t}$
  are the forward difference operators on $j'$ and $t$, respectively (see Appendix~\ref{S:A_GZ}).
  
  The certificates are
  \begin{align*}
    \R_{j'} &= 0,  \\
    \R_{t} &= \frac{ (s-j'-t+1) (s+t-1) } { (i (-j'+i+1-s-t)) } .  
  \end{align*}
  
  By summing the variables $j'$ and $t$ on the left side
  of the recurrence, and by
  telescoping on the right side, the double sum
  satisfies a first order inhomogeneous recurrence
  \begin{equation*} \label{E:Q3_new_double_sum_recur} 
    -\D(s, i, j) + \D(s, i+1, j) = \Inhom(s, i, j),  
  \end{equation*}
  where $\Inhom(s, i, j)$ is a single sum.
  The inhomogeneous term $\Inhom(s, i, j)$ is evaluated in Proposition~\ref{P:Q3_inhom}.
\end{proof}

Since Proposition~\ref{P:Q3_inhom}
demonstrates that when $i>j$, the inhomogeneous term vanishes,
by the recurrence of Eq.~\eqref{E:Q3_double_sum_recur}
in Proposition~\ref{P:Q3_double_sum_recur}
$\D(s, i, j)$ is independent of $i$.
Below we show explicitly that when $i=2s$,
the double sum $\D(s, 2s, j)$ equals to $(-1)^{j+1}\binom{2s}{j}$,
hence the double sum $\D(s, i, j)$ takes this value when $i>j$.
\begin{lemma} [Value of the double sum at $i=2s$]
  When $i=2s$,
  \[
    \D(s, 2s, j) = (-1)^{j+1}\binom{2s}{j} .
  \]
\end{lemma}

\begin{proof}
  When $i=2s$,
  the inner sum of
  Eq.~\eqref{E:Q3_double_sum_new}
  is \Go-summable,
  with the \anti as
  \[
    \frac{ (-1)^{s+j+j'+t} (s+t-1) }
    {t}
    \binom{-j'+ 2 s-1}{ s+t-1 }
    \binom{s} { t-1 }
    \binom{s} {j-j')}
    { \binom{s-j'} {t} }^{-1}.
  \]
  By using this expression of the \anti, the inner sum of $t$ is evaluated as
  \[
    (-1)^{j+1}
    \binom{s} { j'}
    \binom{s} {j-j'}
    +
    (-1)^{s+j'+j}
    \binom{s} {j-j'}
    \binom{-j'+2s-1} {s-1} .
  \]
  The sum of the first term over $j'$ is evaluated as
  \[
    (-1)^{j+1} \binom{2s} { j}
  \]
  by using identity Eq.~\eqref{E:binom_iden_V}.
  We show that the sum over $j'$  of the second term in the \anti vanishes.

  It turns out that the second term is also \Go-summable \wrt $j'$,
  with the antidifference as
  \[
    (-1)^{s+j'+j} 
    \frac{    s}
    {j-2s }
    \binom{2s - j'} {s}
    \binom{s-1} {j-j'}  .
  \]
  This antidifference vanishes at both ends at $j'=s+1$ and $j'=0$
  when $j \ge s$,
  hence the sum is zero.
  
\end{proof}

From this Lemma and Proposition~\ref{P:Q3_double_sum_recur} we prove that
the double sum $\D(s, i, j)$ takes the value of $(-1)^{j+1}\binom{2s}{j}$
in the
strict upper-left triangle of the third \qu.
\begin{coro}[Double sum in the strict upper-left triangle] \label{C:double_sum}
  The double sum in the strict upper-left triangle,
  where $i>j$,
  has the following value
  \[
    \D(s, i, j) = (-1)^{j+1}\binom{2s}{j} .
  \]
\end{coro}

From Proposition~\ref{P:Q3_S1X} and Corollary~\ref{C:double_sum}
we see that $\btwo_{i,j}^{(3)}$ vanishes in the strict upper-left triangle of the third \qu.
\begin{prop} \label{P:Q3_third}
  In the strict upper-left triangle of the third \qu,
  \[
    \btwo_{i,j}^{(3)} = 0.
  \]
\end{prop}

Now we deal with the inhomogeneous term $\Inhom(s, i, j)$
of the recurrence relation Eq.~\eqref{E:Q3_double_sum_recur}.
As mentioned in Appendix~\ref{S:A_GZ}, 
for multi-variable multisum with
nonstandard boundary conditions,
the application of \Ze's algorithm becomes more complicated.
For double sums with nonstandard boundary conditions,
two single sums will appear in the inhomogeneous terms
(see the Lemmas~\ref{L:inhomo1}, \ref{L:inhomo2} and ~\ref{L:inhomo3}
in Appendix~\ref{S:A_GZ}).
If the running index is the same as the parameter of the recurrence,
additional single sums will be needed to cover the elements on the boundary.

In our case here we will use $i$ as the recurrence parameter.
Based on the relation of the summation indices $t$ and $j'$
of the double sum $\D(s, i, j)$ (Eq.~\eqref{E:Q3_double_sum_new}),
Lemma~\ref{L:inhomo3} will be used in the following.
Since one of the certificates $\R_{j'} = 0$,
we need to only work on the single sum related to $\R_t$
(Proposition~\ref{P:Q3_double_sum_recur}).

After putting back the factor $(-1)^{s+j+1}$, the inhomogeneous term of the
recurrence in Proposition~\ref{P:Q3_double_sum_recur}
is given below.
The case of $i=j$ will be used in the next section
when we discuss the terms on the diagonal (section~\ref{SS:Q3_diag}).
\begin{prop}[Inhomogeneous term]  \label{P:Q3_inhom}
  \[
    \Inhom(s, i, j) = \begin{dcases}
      (-1)^{j+1} \binom{2s}{i} ,  & \quad s \le i = j, \\
      0, & \text{otherwise} .
      \end{dcases}
  \]
\end{prop}

\begin{proof}
  Denote the summand of the double sum of Eq.~\eqref{E:Q3_double_sum_new}
  by $F$.
  
  Using the notation of Lemma~\ref{L:inhomo3},
  for the double sum $\D(s, i, j)$,
  we have
  \begin{align*}
    u_{j'} = 0, &  \quad v_{j'} = s,\\
    u_t = 1.    
  \end{align*}
  
  When $t=1$,
  \begin{equation} \label{E:Q3_Gt1}
    \G_t(i, j', 1) =  \R_t \F\Big\rvert_{t=1}  = (-1)^{i+j'+1}
    \binom{2s}{i}
    \binom{i-j'-1}{s-1}
    \binom{s}{j-j'} ,
  \end{equation}
  and when $t=s+1-j'$,
  \[
    \G_t(i, j', s+1-j') = \R_t \F\Big\rvert_{t=s+1-j'}  = 0.
  \]
  Hence
  the inhomogeneous term is
  \[
    \Inhom(s, i, j) = -\sum_{j'=0}^{s}   \G_t(i, j', 1).
  \]
  The summand $ \G_t(i, j', 1)$ is Gosper-summable,
  with antidifference $W(j')$ as
  \[
    W(j') =
    (-1)^{i+j'}
    \binom{2s}{i}
    \binom{i-j'}{ i-j}
    \binom{i-j-1}{i-s-j'} .
  \]

  We discuss two cases $i>j$ and $i=j$ separately
  for $W(j')$.  The case of $i=j$ will be used in the discussions of the terms
  along the diagonal.
  
  When $i>j$,
  for $W(j')$ to be nonzero,
  the second binomial coefficient needs $j' \ge j$,
  while the third binomial coefficient needs $-j-1 \ge j' - s$,
  which is $j \le j'-s+1$.
  But we have the range for $j'$ and $j$ as $j' \le s \le j$,
  so
  $j' = j = s$
  is the only choice;
  this leads to $s \le 1$;
  so the only combinations for $W(j')$ to be nonzero
  is
  when $(s=1, i=2, j=1, j'=1)$.
  In this case $W(j') = -1$, a constant.
  Hence
  When $i>j$,
  $\Inhom(s,i,j) = 0$.

  For $i=j$, there is a pole in the  antidifference $W(j')$,
  so we use \Ze's algorithm to do the sum.
  The summand $\G_t(i, j') = \G_t(i, j', 1)$ when evaluated at $i=j$ is
  \[
    \G_t(i, j')  = (-1)^{i+j'+1}
    \binom{2s}{i}
    \binom{i-j'-1} {s-1}
    \binom{s}{i-j'},
  \]
  which obeys a first order homogeneous recurrence
  \[
    (-2s+i) \G_t(i, j') + (i+1) \G_t(i+1, j') = 0, 
  \]
  with the certificate as
  \[
    \R = \frac{ (i^2+3 i-i s-3 s+j' s-2 i j'+j'^2-2 j'+1) (i-j') (i-j'-s) }
    {(i-j'-s+1) (i-j'+1) }.
  \]

  The sum $\Inhom(s, i, i) = \Inhom(s, i) = \sum_{j'=0}^{s} \G_t(i, j')  $ can be obtained as
  \begin{align*}
    \Inhom(s, i) &= \frac{2s -i +1}{i} \Inhom(s, i-1) \\
             &= \frac{2s -i +1}{i} \frac{2s -i + 2}{i-1} \cdots \frac{s}{s+1} \Inhom(s, s) \\
    &= \frac{s!^2}{i! (2s-i)!} \Inhom(s, s) .
  \end{align*}

  When $i=s$, $\Inhom(s,s)$ can be easily obtained, since only one nonvanishing term remains
  for the sum when $j'=0$:
  \[
    \Inhom(s,s) = (-1)^s \binom{2s}{s} .
  \]
  Hence we have for $i=j$,
  \[
    \Inhom(s,i) = (-1)^s \binom{2s}{i} .
  \]
\end{proof}

\subsection{The diagonal of the third \qu $Q_3$} \label{SS:Q3_diag}

We show in this section that overall contributions from the recurrence relations in the set \PTwo
in the third \qu diagonal is zero.
Since we have already shown that the recurrence identities in the set \POne
give the correct coefficients along the diagonal (by a factor of $2$) (Corollary~\ref{C:P1_diag}),
the main theorem is proved for the diagonal in the third \qu.

First,
from Eq.~\eqref{E:Q3_s1g1} we see that
when $i=j$, the boundary condition of the summation
of $\btwo_{i,j}^{(1)}$
becomes standard,
so the  identity Eq.~\eqref{E:binom_iden_V} can be used to evaluate the sum.
Hence the contribution from the first term of $\coe_{i,j}$
of the vertical recurrence identities in the set \PTwo is
\begin{equation} \label{E:Q3_diag_g1}
    \btwo_{i, i}^{(1)} = (-1)^{i+1} \binom{2s}{i} .
\end{equation}
The horizontal recurrences are not long enough to
touch the diagonal in the third \qu (Figure~\ref{F:S2} and Figure~\ref{F:S2s5}).

We  have also proved that the contribution from the second term of $\coe_{i,j}$
of the recurrence identities in the set \PTwo vanishes in the diagonal
(Lemma~\ref{L:Q3_S10}).

For the third term of $\coe_{i,j}$,
we first show that the double sum in Proposition~\ref{P:Q3_S1X} (the correction term)
vanishes in the diagonal.

\begin{lemma}[Double sum vanishes]
  When $i=j$,
  \[
    \D(s, i, i) = 0 .
  \]
\end{lemma}

\begin{proof}
  The double sum $\D(s, i, j)$ follows the recurrence of Eq.~\eqref{E:Q3_double_sum_recur}
  on the diagonal,
  \[
    -\D(s, i, i) + \D(s, i+1, i) = (-1)^{i+1} \binom{2s}{i},
  \]
  but we have shown that $\D(s, i+1, i) = (-1)^{i+1} \binom{2s}{i}$ in Corollary~\ref{C:double_sum},
  hence $\D(s, i, i) = 0$ .
\end{proof}

From the above Lemma we see the contribution of the third term of $\coe_{i,j}$
of the recurrence identities in the set \PTwo is given by the first term
of Eq.~\eqref{E:Q3_S1X},
\begin{equation} \label{E:Q3_diag_g3}
    \btwo_{i, i}^{(3)} = (-1)^{i} \binom{2s}{i} .
\end{equation}

Putting
Eq.~\eqref{E:Q3_diag_g1},
Eq.~\eqref{E:Q3_S10c},
and
Eq.~\eqref{E:Q3_diag_g3}
together, we have
\begin{prop} \label{P:Q3_diag_P2}
  In the diagonal of the third \qu,
  the contributions from the recurrence identities in the set \PTwo
  is zero.
  \[
    \btwo_{i,i} = \btwo_{i,i}^{(1)}
    + \btwo_{i,i}^{(2)}
    + \btwo_{i,i}^{(3)}
    = 0.
  \]
\end{prop}

\begin{proof}[Partial proof of the main theorem in the third \qu{}]
  By combining
  Lemma~\ref{L:Q3_P1vP2_1},
  Corollary~\ref{C:Q3_P1hP2_2},
  and
  Proposition~\ref{P:Q3_third},
  we prove that in the strict upper-left triangle of the third \qu,
  the terms from the recurrence identities from the sets \POne and \PTwo
  cancel each other out.
  By symmetry,
  this also holds for the strict lower-right triangle in the third \qu.
  By combining
  Corollary~\ref{C:P1_diag} and
  Proposition~\ref{P:Q3_diag_P2}
  we prove that the terms in the diagonal  of the third \qu
  have the correct values of the main theorem, by a factor of two.
  Hence the left side of the main theorem Eq.~\eqref{E:goal_lhs} is
  proved in the third \qu.
\end{proof}

\subsection{The fourth \qu $Q_4$} \label{S:Q4}

\subsubsection{Summary}

In the fourth \qu $Q_4$ (top left \qu),
$0 \le i \le s$ and
$s \le j \le 2s$.
As for the third \qu,
the situation is also complicated, but in a different way from the third \qu,
since at each lattice site,
both horizontal and vertical recurrence identities from both \POne and \PTwo
contribute
(Figure~\ref{F:S1}, Figure~\ref{F:S1s5},  Figure~\ref{F:S2}, and Figure~\ref{F:S2s5}).

The site $(s, s)$ is discussed in the $Q_3$ section, and it is excluded
from the following discussions of $Q_4$.

At each position $(i,j)$ in the fourth \qu $Q_4$,
where $0 \le i \le s$ and $s \le j \le 2s$,
the contributions from \POne are
\begin{align} \label{E:Q4_S1}
  \bone_{i,j}^{v} +  \bone_{i,j}^{h} &=
      (-1)^j  \sum_{j'=\max(0, i-s,j-s)}^{ \min(s, i, j)  }  \binom{s}{j'}  \binom{s}{j-j'}
    +
    (-1)^i  \sum_{i'=\max(0, i-s,j-s)}^{ \min(s, i, j)  }  \binom{s}{i'}  \binom{s}{i-i'}    \nonumber \\
                                                    &= (-1)^j  \sum_{j'=j-s}^i \binom{s}{j'}  \binom{s}{j-j'}
                                                      + (-1)^i  \sum_{i'=j-s}^i \binom{s}{i'}  \binom{s}{i-i'} .
\end{align}

The  contributions from \PTwo are
\begin{align} \label{E:Q4_P2}
  \begin{aligned}
  \btwo_{i,j}^{v}
  &= \sum_{j'=\max(i+1, j-s)}^{s} (-1)^{j-j'} \coe_{i, j'} \binom{s}{j-j'}
                            = \sum_{j'=0}^{s} (-1)^{j-j'} \coe_{i, j'} \binom{s}{j-j'} \\
  \btwo_{i,j}^{h}
  &= (-1)^s \sum_{i'=0}^{i}    (-1)^{i-i'} \coe_{2s-j, s-i'}  \binom{s}{i-i'}
                            = (-1)^s \sum_{i'=0}^{j-s}    (-1)^{i-i'} \coe_{2s-j, s-i'}  \binom{s}{i-i'}.
                          \end{aligned}
\end{align}

We will prove the following identities.
The contributions from the first term of the vertical (horizontal) recurrence identities in \PTwo
cancel those from the vertical (horizontal) identities in \POne,
\begin{align} \label{E:Q4_1st}
  \begin{aligned}
  \bone_{i,j}^v + \btwo_{i,j}^{1v} &= 0,\\
  \bone_{i,j}^h + \btwo_{i,j}^{1h}  &= 0 .
  \end{aligned}
\end{align}

The contributions from the second term of the vertical (horizontal) recurrence identities in \PTwo
cancel those from the third term of the horizontal (vertical) identities in \PTwo,
\begin{align} \label{E:Q4_2nd_3rd}
  \begin{aligned}
  \btwo_{i,j}^{2v} +  \btwo_{i,j}^{3h} &= 0, \\
  \btwo_{i,j}^{2h} +  \btwo_{i,j}^{3v} &= 0.
  \end{aligned}
\end{align}

Combining Eq.~\eqref{E:Q4_1st} and Eq.~\eqref{E:Q4_2nd_3rd} we prove that
\[
  \bone_{i,j} + \btwo_{i,j} = 0
\]
in the fourth \qu.

\subsubsection{The contributions from the first term of $\coe_{i,j}$}

\begin{lemma}[The first term of $\coe_{i,j}$] \label{L:Q4_1st}
The contributions from the first term of the vertical (horizontal) recurrence identities in \PTwo
cancel those from the vertical (horizontal) identities in \POne (Eq.~\eqref{E:Q4_1st}).
\end{lemma}

\begin{proof}
  The contribution from the vertical identities of \POne
  is
  \begin{equation} \label{E:Q4_s1g1}
    \bone_{i,j}^v = 
    (-1)^j  \sum_{j'=\max(0, i-s,j-s)}^{ \min(s, i, j)  }  \binom{s}{j'}  \binom{s}{j-j'}
    = (-1)^j  \sum_{j'=j-s}^i \binom{s}{j'}  \binom{s}{j-j'},
  \end{equation}
  while the  contribution from horizontal identities of \POne is
  \begin{equation} \label{E:Q4_s2g1}
    \bone_{i,j}^h = 
    (-1)^i  \sum_{i'=\max(0, i-s,j-s)}^{ \min(s, i, j)  }  \binom{s}{i'}  \binom{s}{i-i'}
    =  (-1)^i  \sum_{i'=j-s}^i \binom{s}{i'}  \binom{s}{i-i'} .
  \end{equation}
  
  The contribution of the first term of $\coe_{i, j'}$ in the sum of $\btwo_{i,j}^{v}$  is simplified as
  \begin{equation*}
    \btwo_{i,j}^{1v} = 
   \sum_{j'=j-s}^{i} (-1)^{j-j'} (-1)^{j'+1} \binom{s}{j'} \binom{s}{j-j'} 
  = (-1)^{j+1} \sum_{j'=j-s}^{i}  \binom{s}{j'} \binom{s}{j-j'},
\end{equation*}
%which cancels the first term in Eq.~\eqref{E:Q4_s1g1}.
which cancels the Eq.~\eqref{E:Q4_s1g1}.

The first term of $\coe_{2s-j, s-i'}$ in the summand of $\btwo_{i,j}^{h}$  is
\[
  (-1)^{s-i'+1} \binom{s}{s-i'} [s-i' \le 2s-j],
\]
which is simplified in the sum as
\begin{align*}
  \btwo_{i,j}^{1h} &= 
   (-1)^s \sum_{i'=0}^{i} (-1)^{i-i'} (-1)^{s-i'+1} \binom{s}{s-i'} [s-i' \le 2s-j] \binom{s}{i-i'}  \\
  &= (-1)^{i+1} \sum_{i'=j-s}^{i}  \binom{s}{i'} \binom{s}{i-i'}.
\end{align*}
%This cancels the second term in Eq.~\eqref{E:Q4_s2g1}.
This cancels the Eq.~\eqref{E:Q4_s2g1}.

\end{proof}

\subsubsection{The contributions from the second term $\btwo_{i,j}^{2h}$
cancels those from the third term $\btwo_{i,j}^{3v}$}

\begin{prop} \label{P:Q4_2ndh_3rdv}
  In the fourth \qu $Q_4$,
  the contributions from the second term of the horizontal recurrence identities in \PTwo
cancel those from the third term of the vertical identities in \PTwo
  \[
     \btwo_{i,j}^{2h} +  \btwo_{i,j}^{3v} = 0.
  \]
  where
  \begin{align*}
    \btwo_{i,j}^{2h} &=
    (-1)^{s+i+j}
    \binom{2s}{j}
    \sum_{i'=0}^{s}  (-1)^{i'} \binom{j-i'-1}{s}  \binom{s} {i-i'}, \\
    \btwo_{i,j}^{3v} &= 
    (-1)^{i+j+1}
    (2s - i) \binom{2s}{i}
    \sum_{j'=0}^{s}
    \sum_{t=0}^{i}  (-1)^{t+j'}  
    \frac{ 1} {2s - t }
    \binom{i}{t}
    \binom{s-t-1+j'}{ j'}
    \binom{s}{j-j'} .
  \end{align*}
\end{prop}

To prove Proposition~\ref{P:Q4_2ndh_3rdv}
we would like to express the \emph{inner sum} \wrt $t$ in $\btwo_{i,j}^{3v}$
in two parts, a simple closed form term and a second correction sum,
as we did for the third \qu $Q_3$. 
Unlike the third \qu, the correction term is for $j'>i$.

Below we first make this transform for the inner sum of $\btwo_{i,j}^{3v}$.
\begin{lemma}[Transform of the inner sum] \label{L:Q4_inner_sum}
  The inner sum of $\btwo_{i,j}^{3v}$ \wrt $t$ can be rewritten as
  \begin{multline*} 
          \sum_{t=0}^{i}  (-1)^{t}
        \frac{ 1} {2s - t }
        \binom{i}{t} \binom{s-t-1+j'}{ j'}
        =   \nonumber \\
        \frac{(-1)^{i+j'}}{2s - i}
        { \binom{s}{j'} } { \binom{2s}{i}}^{-1}
        +
         \sum_{t=1}^{j'}
    \frac{  (-1)^{t+1}  }{t}         %%%%%%%%%% change this one: why not in the test?
    \binom{s}{ t-1}
    \binom{s+j'-i-1}{ s+t-1}
    { \binom{j'}{ t} }^{-1}.
  \end{multline*}
\end{lemma}

\begin{comment}
  
The overall factor:
Q3:
\[
  s+i+j+j',
\]
Q4:
\[
  i+j+1.
\]
In Maple, the $+1$ absorbed $t+1$.
The factor was written as $i+j$.

\end{comment}

\begin{proof}
  Use the following changes of variables on Eq.~\eqref{E:Q3_sum_t_corr}:
  \begin{align} \label{E:Q4:change}
    \begin{aligned}
    i  & \to  2s -i, \\
    j' & \to s - j' .
    \end{aligned}
    \end{align}        
\end{proof}

Now the sum $\btwo_{i,j}^{3v}$
can be written in two parts:
\begin{lemma} \label{L:Q4_btwov}
  \begin{multline*}
    \btwo_{i,j}^{3v}
    = (-1)^{j+1}
    \binom{2s}{j} \\
    + 
    (-1)^{i+j}  (2s-i)
    \binom{2s}{ i}
    \sum_{j'=0}^{s}  
    \sum_{t=1}^{j'}
    \frac{  (-1)^{j'+t}  }{t}
    \binom{s}{ t-1}
    \binom{s+j'-i-1}{ s+t-1}
    \binom{s} {j-j'}
    { \binom{j'}{ t} }^{-1}.
  \end{multline*}
\end{lemma}
\begin{proof}
  The outer sum \wrt $j'$ can be similarly carried out on the first term of Lemma~\ref{L:Q4_inner_sum}
  as in Lemma~\ref{L:Q3_inner_sum_1st_term}.
\end{proof}

\begin{comment}

Hence we need to prove
\[
  a_{n-i, m-j}^{(S_{2v})' } + a_{n-i, m-j}^{(S_{2h}) '} = 0,
\]
where
\begin{align*}
  a_{n-i, m-j}^{(S_{2v})' } &= \sum_{j'=j-s}^{s} (-1)^{j-j'} \coe_{i, j'}' \binom{s}{j-j'}, \\
  a_{n-i, m-j}^{(S_{2h}) '} &=  (-1)^s \sum_{i'=0}^{i}    (-1)^{i-i'} \coe_{2s-j, s-i'}'  \binom{s}{i-i'} ,
\end{align*}
and here
$\coe'$ is $\coe$ without the first term (~Eq.~\eqref{E:hv2}).

And we have
\begin{align}
  &\binom{2s}{i} \sum_{j'=j-s}^{s} (-1)^{j'}  \binom{s-i-1+j'}{s} \binom{s}{j-j'}
  + (-1)^s \frac{j}{2s-j+1} \binom{2s}{j}
       \sum_{i'=0}^{i} \sum_{t=0}^{2s-j+1} (-1)^{t+i'}
       \frac{ t } {2s - t + 1 }
       \binom{2s-j+1}{t}  \binom{2s-t-i'} {s-i'}  \binom{s} {i-i'}
       = 0,  \label{E:S2vS2h_4I}  \\
  &     (-1)^s  \binom{2s}{j}  \sum_{i'=0}^{i}  (-1)^{i'}
    \binom{j-i'-1}{s}  \binom{s} {i-i'}
       + \frac{ 2s - i }{i+1} \binom{2s}{i}
    \sum_{j'=j-s}^{s} \sum_{t=0}^{i+1}  (-1)^{t+j'}  
    \frac{ t } {2s - t + 1 }
    \binom{i+1}{t} \binom{s-t+j'}{ j'} \binom{s}{j-j'}
    = 0 .  \label{E:S2vS2h_4II}
\end{align}

    Eq.~\eqref{E:S2vS2h_4I} can be transformed to Eq.~\eqref{E:S2vS2h_4II}
    by making the changes
    \begin{align}
      i &\rightarrow 2s - j, \\
      i' &\rightarrow s - j' .
      \end{align}
      Hence if we prove one identity, the other is also proved.
\end{comment}

We prove Proposition~\ref{P:Q4_2ndh_3rdv}
by showing that
$\btwo_{i,j}^{2h}$ and
$-\btwo_{i,j}^{3v}$
satisfy the same recurrence and have the same initial conditions.
The following Lemma shows the recurrence satisfied by $\btwo_{i,j}^{2h}$.

\begin{lemma} [Recurrence of $\btwo_{i,j}^{2h}$] \label{L:Q4_S20_recur}
  The term $\btwo_{i,j}^{2h}$ satisfies the following first order inhomogeneous recurrence
  \wrt $i$:        
\begin{equation}  \label{E:Q4_S20_recur}
  \btwo_{i+1,j}^{2h} - \btwo_{i,j}^{2h}
  =
   \left( -1 \right) ^{i+j+s+1}{2\,s\choose i+1}{2\,s-i-1\choose j-i-1}{
j-i-2\choose j-s-1} .
 \end{equation}
\end{lemma}

\begin{proof}
  The summand of $\btwo_{i,j}^{2h}$
  \[
    \F(i, j') = (-1)^{s+i+j+j'}
    \binom{2s} { j}
    \binom{j-i'-1}{s}
    \binom{s} {i-i'}
  \]
  obeys a first order recurrence shown by the \Ze's algorithm
  \[
    (-i-1+j) \F(i, j') + (i-j+1) \F(i+1, j')
    = \G(i, j'+1) - \G(i, j'),
  \]
  with the certificate as
  \[
    \R = \frac{ (j-j') (-s+i-j') } { i+1-j' } .
  \]
  Sum up with $j'$ from $0$ to $s$ by telescoping
  to obtain Eq.~\eqref{E:Q4_S20_recur},
  as the value of $\G(i, s+1) - \G(i, 0) = - \G(i, 0)$,
  since $\G(i, s+1) = 0$.
\end{proof}

The recurrence satisfied by $\btwo_{i,j}^{3v}$ is given below.
\begin{lemma} [Recurrence of $\btwo_{i,j}^{3v}$] \label{L:Q4_S1_recur}
  The double sum in $\btwo_{i,j}^{3v}$ of Lemma~\ref{L:Q4_btwov} satisfies the following first order inhomogeneous recurrence \wrt $i$:
\begin{equation}  \label{E:Q4_S1_recur}
  \btwo_{i+1,j}^{3v} - \btwo_{i,j}^{3v}
  =
   \left( -1 \right) ^{i+j+s}{2\,s\choose i+1}{2\,s-i-1\choose j-i-1}{
j-i-2\choose j-s-1} .
 \end{equation}
  
\end{lemma}

\begin{proof}
  The summand $\F$ of the double summation in $\btwo_{i,j}^{3v}$
  \[
    \F(i; j', t) =
     (-1)^{i+j+j'+t}  
    \frac{ 2s-i }{t}
    \binom{2s}{ i}
    \binom{s}{ t-1}
    \binom{s+j'-i-1}{ s+t-1}
    \binom{s} {j-j'}
    { \binom{j'}{ t} }^{-1}
  \]
  satisfies a first order inhomogeneous recurrence by the multi-variable multisum version of \Ze's algorithm
  \[
    -\F(i, j', t) + -\F(i+1, j', t) = \Inhom(s,i,j),
  \]
  with the certificates
  \begin{align*}
    \R_{j'} &= 0, \\
    \R_{t} &=
             \frac{ -(j'-t+1) (s+t-1) } { (-s-j'+i+1) (i+1) }.
  \end{align*}

  Sum up both sides on $j'$ and $t$,
  the right side telescopes, with two single sums left (Lemma~\ref{L:inhomo2} in Appendix~\ref{S:A_GZ}).
  In this case, since $ \R_{j'} = 0$
  and $\R_{t} \F (t=j'+1) = 0$,
  there is only one single sum left:
  \begin{align*}
    RHS &= -\sum_{j'=1}^{s} \R_{t} \F (t=1)  \\
    &= (-1)^{i+j} \sum_{j'=1}^{s}  (-1)^{j'}
    \frac{s (i-2 s) }
     { (-s-j'+i+1)(i+1) }
     \binom{2s}{ i}
     \binom{s+j'-i-1}{ s}
     \binom{s} {j-j'} .
  \end{align*}

  The summand is \Go-summable,
  with the antidifference as
  \[
    (-1)^{i+j+j'+1}
    \binom{2s}{ i+1}
    \binom{s+j'-i-2} {j-i-1}
    \binom{j-i-2} {j-j'} .
  \]
  Since the antidifference vanishes at $j'=1$,
  set $j'=s+1$ in the antidifference to get the sum,
  which proves Eq.\eqref{E:Q4_S1_recur}.
\end{proof}

\begin{proof}[Proof of Proposition~\ref{P:Q4_2ndh_3rdv}]
  
  Lemma~\ref{L:Q4_S20_recur} and
  Lemma~\ref{L:Q4_S1_recur}
  show that
  $\btwo_{i,j}^{2h}$ and
  $-\btwo_{i,j}^{3v}$ satisfy the same first order inhomogeneous recurrence.
  We only need to check that $\btwo_{i,j}^{2h}$ and
  $-\btwo_{i,j}^{3v}$ share the same initial value.
  At $i=0$,
  \[
   \btwo_{0,j}^{2h} = (-1)^{s+j} \binom{2s}{j} \binom{j-1}{s} ,
\]
while the value of $\btwo_{i,j}^{3v}$ when $i=0$ is
\[
  \btwo_{0,j}^{3v}
  = (-1)^{j} \sum_{j'=0}^{s}   (-1)^{j'}  
  \binom{s-1+j'}{ j'} \binom{s}{j-j'}
  =  (-1)^{s+j+1} \binom{2s}{j} \binom{j-1}{s} = -\btwo_{0,j}^{2h}.
\]
This proves Proposition~\ref{P:Q4_2ndh_3rdv}.
\end{proof}

We summarize the results for the fourth \qu.
\begin{proof}[Partial proof of the main theorem in the fourth \qu{}]
  Lemma~\ref{L:Q4_1st} proves the identities of Eq.~\eqref{E:Q4_1st},
  and Proposition~\ref{P:Q4_2ndh_3rdv} proves
  the second identity of \eqref{E:Q4_2nd_3rd}.
  By symmetry, the first identity of \eqref{E:Q4_2nd_3rd}
  can be proved in a similar way.
  Put these results together, the second identity
  of Eq.~\eqref{E:goal_lhs} is proved
  in the fourth \qu.
\end{proof}

By symmetry, a similar conclusion is reached for the second \qu $Q_2$.

\subsection{The right hand side} \label{S:RHS}
In the previous sections we investigated the values on the left side of equations 
when the recurrence identities from the sets \POne and \PTwo are added with proper weights.
In this section we calculate the sum of these numbers on the right hand side.
By symmetry we only need to look at the vertical contributions from
\POne and \PTwo. In the following the superscript $v$ is omitted. 

For two-dimensional lattice strips, the right hand side of the recurrences
depend on the boundary conditions of the horizontal (vertical) direction
if the recurrence is in the vertical (horizontal) direction.
For recurrence in the vertical direction,
the right hand side is
$(2n + \text{constant})^s$ (Eq.~\eqref{E:rec1})~\cite{Kong2024}.
Below we generalize slightly and assume the right hand side of the vertical recurrence is
\[
  (\ca n + \cb)^s.
\]
We will show that for the diagonal recurrence in the square,
the right hand side depends only on $\ca$ and $s$,
and not on $n$, $m$, nor $\cb$.

We discuss the sets \POne and \PTwo separately.

\subsubsection{The right hand side from \POne}

For vertical recurrences, the right hand side constant
is the same for each \emph{column} (they has the value of lattice width $n$).
The following Lemma gives the sum of the coefficients $\cob_{i,j}$ of the recurrence
identities in \POne at a given column $i$ (Figure~\ref{F:S1} and Figure~\ref{F:S1s5},  red arrowheads),
$x(i)$ (for $0 \le  i \le s$) and $x'(i)$ (for $s <  i \le 2s$).
\begin{lemma} \label{L:RHS_P1_sum_coe} 
  \begin{align}  \label{E:RHS_P1_sum_of_col}
    \begin{aligned}
    \x(i) = \sum_{j'=0}^{i}
    (-1)^{j'}
    \binom{s}{j'}
    &=
      (-1)^i \binom{s-1}{ i} , \quad 0 \le  i \le s, \\
   \x'(i) =  \sum_{j'=i-s}^{s} (-1)^{j'}
    \binom{s}{j'}
    &=
      (-1)^{s+i} \binom{s-1}{ 2s - i} ,  \quad s <  i \le 2s.
          \end{aligned}
      \end{align} 
\end{lemma}
\begin{proof}
  The summand of the first equation is \Go-summable,
  with the \anti shown in Lemma~\ref{L:anti2}.
  
  The second identity can be obtained from the first
  by changing the variable
  $i \to 2s -i$ and $j' \to s - j'$.
\end{proof}

Hence for vertical recurrences from \POne,
at column $n-i$  when $0 \le i \le s$, the contributions from the right side are
\[
  \X(i) =
  c(n-i)^s
  \x(i)
  = (-1)^i \binom{s-1}{ i}  c(n-i)^s.
  \]
Note that when $i=s$, $X(s) = 0$.
  
When   $s <  i \le 2s$,
\[
  \X'(i) =
  c(n-i)^{s}
  \x'(i)
  = (-1)^{s+i} \binom{s-1}{ 2s - i}  c(n-i)^s .
\]

The total contributions of verticals from \POne (Figure~\ref{F:S1} and Figure~\ref{F:S1s5}, red arrowheads)
are the sum of $\X(i)$ and $\X'(i)$ over all the columns in the square:
\[
  \sum_{i'=0}^{s-1}  \X(i')
  + \sum_{i'=s+1}^{2s}  \X'(i') .
\]

We are going to calculate this sum.
\begin{prop}[The right hand side of vertical recurrences in \POne] \label{P:RHS_P1_V_gen}
  For $c(n) = \ca n + \cb$,
   \begin{equation} \label{E:RHS_P1_V_gen}
    \sum_{i'=0}^{s-1}  \X(i') + \sum_{i'=s+1}^{2s}  \X(i') = \ca^s (s+1)! .
  \end{equation}
\end{prop}

First we prove the following identities:
\begin{lemma} \label{L:RHS_Stirling}
  \begin{equation}
    \sum_{i=0}^{s} (-1)^{i+s} i^{s+1-t}
    \binom{s}{i}
    =
    \begin{cases}
       \frac{1}{2} s  (s+1)! , & \quad t=0, \\
       s! ,                  , & \quad t=1, \\
       0,                    , & \quad t>1.
     \end{cases}
     \end{equation}
\end{lemma}

\begin{proof}
  These identities are special cases of
  a known identity involving Stirling number of the second kind:
  \[
    \sum_{i=0}^{s} (-1)^{i+s} i^{p} \binom{s}{i}
    = s! \stirling{p}{s},
  \]
  where $\stirling{p}{s}$ is the Stirling number of the second kind
  \cite[Eq.~(6.19)]{Graham1994}.
  Let $p=s+1$ in above identity and use the fact that
  \[
    \stirling{s+1}{s} = \binom{s+1}{2}
  \]
  to obtain the first identity for $t=0$.
  
  Let $p=s$ and use
  \[
    \stirling{s}{s} = 1
  \]
  to obtain the second identity  for $t=1$.

  When $t>1$, we have $s+1-t < s$. In this case the Stirling number of the second kind vanishes.
  
\end{proof}

With Lemma~\ref{L:RHS_Stirling} we proceed to prove Proposition~\ref{P:RHS_P1_V_gen}.
\begin{proof}[Proof of Proposition~\ref{P:RHS_P1_V_gen}]
  By making a change of summation variable in the second sum $i' \to 2s -i'$,
  the two sum in Eq.\eqref{E:RHS_P1_V_gen} can be combined into a single sum
  \begin{equation}  \label{E:RHS_P1_total}
    \sum_{i'=0}^{s-1}  \X(i')
    + \sum_{i'=s+1}^{2s}  \X'(i')
    =
    \sum_{i'=0}^{s-1}
      (-1)^i
      \binom{s-1} {i}
      \left[  c(n-i)^s + \left( -c(n-2s+i) \right)^s \right] .
  \end{equation}

  If we write
  $
    c_0 = c(n),
  $
  then
  \begin{align*}
    c(n-i) &= c_0 - \ca i, \\
    -c(n-2s+i) &= 2 \ca s - c_0 - \ca i.
  \end{align*}
  Expand $ c(n-i)^s $ and $ \left( -c(n-2s+i) \right)^s $ in Eq.~\eqref{E:RHS_P1_total}
  to obtain a double summation
  \[
    \sum_{i=0}^{s-1} (-1)^i
    \binom{s-1}{i}
    \sum_{t=0}^{s} 
    \binom{s}{t}
    \left(
      c_0^t + (2 \ca s-c_0)^t 
    \right)
    (-\ca i)^{s-t} .
  \]

  Lemma~\ref{L:RHS_Stirling}
  shows that for the outer summation over $i$,
  the sum does not vanish
  only for two values of $t$ of the inner sum, i.e., $t=0$ and $t=1$.

  When $t=0$, we have
  \[
  \sum_{i=0}^{s-1} (-1)^{i+s} i^{s}
  \binom{s-1}{i}
  = - \frac{1}{2}
  (s-1) s!
  \]
  by Lemma~\ref{L:RHS_Stirling} after making the variable change
  from $s$ to $s-1$.

  At $t=1$,
  the sum over $i$ takes a value of $(s-1)!$.
  The summation of Eq.~\eqref{E:RHS_P1_total} is thus the sum of these two values
  only, multiplied by the factors in the inner sum that are independent of $i$.
  After bringing back these factors, the two terms for the sum are
  \begin{align*}
    & -  \frac{(s-1) s!}{2} \times 2 \ca^{s} , &\quad t=0, \\
    &  (s-1)! \times 2 \ca^{s} s^2,            &\quad t=1 .
  \end{align*}
  Hence the summation in Eq.~\eqref{E:RHS_P1_V_gen} takes the value as the sum of
  these two terms, which proves Proposition~\ref{P:RHS_P1_V_gen}.
\end{proof}

\subsubsection{The right hand side from \PTwo}

For the recurrence identities in the vertical direction from the set \PTwo,
at column $i$  when $0 \le i \le s$, the contributions to the right hand side are
(Figure~\ref{F:S2} and Figure~\ref{F:S2s5}, red arrowheads)
\[
  \Y(i) =
  c(n-i)^s
  \sum_{j'=i+1}^{s} 
  \coe(i,j') ,
  \]
and  
when   $s <  i \le 2s$,
\[
  \Y'(i) =
  (-1)^s c(n-i)^{s}
  \sum_{j'=0}^{i-s-1}
  \coe(2s-i, s-j').
\]

The total contributions of verticals from \PTwo 
are
\[
  \sum_{i'=0}^{s-1}  \Y(i')
  + \sum_{i'=s+1}^{2s}  \Y(i') .
\]

For the three parts of $\coe_{i,j}$, the first part does not contribute,
since $j'>i$ in the sums of $\Y(i)$ and
$s-j' > 2s - i$ in the sums of $\Y'(i)$.
As for the binomial coefficients $\cob_{i,j}$
in the set \POne (Eq.~\eqref{E:RHS_P1_sum_of_col}),
we will derive analogous sums
for the second and third parts of $\coe_{i,j}$
for each column.

Denote these sums as $\y^{(2)}(i)$ and $\y^{(3)}(i)$ respectively for column $i$.
For the second part of $\coe_{i,j}$,
the column sum of the coefficients in the first \qu is given by
\begin{lemma} \label{L:RHS_P2_h2}
  \begin{equation} \label{E:RHS_P2_h2}
    \y^{(2)}(i) =
    (-1)^{i+1}
    \binom{2s} {i}
    \sum_{j'=i+1}^{s}
    \binom{s+j'-i-1}{ s}
    =
    (-1)^{i+1}
    \binom{2s} {i}
    \binom{2s-i}{ s+1 } .
  \end{equation}
\end{lemma}

\begin{proof}
  The summand is \Go-summable with \anti shown in Lemma~\ref{L:anti1}.
\end{proof}

The column sum of the third part of $\coe_{i,j}$ involves a double sum.
We have for the third part of $\coe_{i,j}$ in the first \qu as
\begin{prop} \label{L:RHS_P2_h3}
  \begin{align} \label{E:RHS_P2_h3} 
    \begin{aligned}
    \y^{(3)}(i)
    &=
    (-1)^{i}  (2s-i)
    \binom{2s} {i} 
    \sum_{t=0}^{i}
    \sum_{j'=i+1}^{s}
    \frac{(-1)^t}{2s - t}
    \binom{i}{ t}
                   \\
    &=
      (-1)^i
      \frac{2s-i}{s}
      \binom{2s} {i}
      \binom{2s -i - 1}{ s}
      \left[
      1
      -  \frac{1}{2}
      { \binom{2s-1}{ s} }^{-1}
    \right] \binom{s-t-1+j'} {j'}.      
    \end{aligned}
  \end{align}
\end{prop}

\begin{proof}
The inner sum of $j'$ in Eq.~\eqref{E:RHS_P2_h3}
is \Go-summable as shown in Lemma~\ref{L:anti1}, hence Eq.~\eqref{E:RHS_P2_h3}
can be written as
\[
  \y^{(3)}(i)
  =
  (-1)^i (2s-i)
  \binom{2s} {i}
  \sum_{t=0}^{i} (-1)^t
  \frac{1}{ 2s - t}
  \binom{i} {t}
  \left[
    \binom{2s-t} {s-t}
    - \binom{s-t+i} {s-t} 
  \right] .
\]

Consider the above equation as a difference between two single sum,
and calculate each sum separately.
For the first sum,
the summand
\[
  \F(i, t) =
  (-1)^t
  \frac{1} { 2s - t}
  \binom{i}{ t}
  \binom{2s-t}{ s-t}
\]
satisfies a first order homogeneous recurrence
\[
  (s-i-1) \F(i, t) +(i-2s+1) \F(i+1, t) = 0,
\]
with the certificate as
\[
  \R = \frac{t (2s-t)} {i-t+1} .
\]
With initial value at $i=0$ as
\[
  \frac{1}{2s}
  \binom{2s}{ s},
\]
the sum can be evaluated as
\begin{equation} \label{E:RHS_h3I}
  (-1)^i (2s-i)
  \binom{2s} {i}
  \sum_{t=0}^{i} (-1)^t
  \frac{1}{ 2s - t}
  \binom{i} {t}
  \binom{2s-t} {s-t}
  =
  \frac{1}{s}
  \binom{2s -i - 1}{ s}.
\end{equation}

For the second sum,
the summand
\[
  \F(i, t) =
  (-1)^t
  \frac{1}{ 2s - t}
  \binom{i} {t}
  \binom{s-t+i} {s-t}
\]
satisfies a first order homogeneous recurrence
\[
  -(i+1)(i-s+1) \F(i, t)  + (i+1)(i-2s+1) \F(i+1, t) = 0,
\]
with the certificate as
\[
  \R = \frac{ t(2s-t)(s-t+i+1)}{i-t+1 }.
\]
With initial value at $i=0$ as
\[
  \frac{1}{2s},
\]
the sum can be evaluated as
\begin{equation} \label{E:RHS_h3II}
  (-1)^i (2s-i)
  \binom{2s} {i}
  \sum_{t=0}^{i} (-1)^t
  \frac{1}{ 2s - t}
  \binom{i} {t}
  \binom{s-t+i} {s-t} 
  =
  \frac{1}{2s}
  \binom{2s-i-1}{ s}
  \binom{2s-1}{ s}^{-1}.
\end{equation}

Eq.~\eqref{E:RHS_P2_h3} is the difference between
Eq.~\eqref{E:RHS_h3I} and
Eq.~\eqref{E:RHS_h3II} .

\end{proof}

Putting together the results for $\y^{(2)}(i)$ (Eq.~\eqref{E:RHS_P2_h2})
and $\y^{(3)}(i)$ (Eq.~\ref{E:RHS_P2_h3}),
we obtain the column sum of the vertical recurrence identities from \PTwo
in the first \qu and the second \qu as
\begin{prop}[Sum of vertical identities in \PTwo at column $i$]  \label{P:RHS_P2_h}
  \begin{align} \label{E:RHS_P2_h}
    \begin{aligned}
    \y(i) = \sum_{j'=i+1}^{s} 
    \coe(i,j')
    =&
    (-1)^i
    \binom{s-1}{ i}
    \left[
      \frac{1}{s} \binom{2s} {s-1}  - 1
       \right], \\
    \y'(i) = \sum_{j'=0}^{i-s-1}
    \coe(2s-i, s-j')
    =&
     (-1)^i
    \binom{s-1}{ 2s - i}
    \left[
      \frac{1}{s} \binom{2s} {s-1}  - 1
       \right],       
    \end{aligned}
    \end{align}
    where the second equation is obtained from the first one by changing $i \to 2s -i$.
\end{prop}

\begin{prop}[Total contributions to the right hand side from
  vertical  identities in \PTwo] \label{P:RHS_P2_V_gen}
  For $c(n) = \ca n + \cb$,
\begin{equation} \label{E:RHS_P2_V_gen}
  \sum_{i'=0}^{s-1}  \Y(i')
  + \sum_{i'=s+1}^{2s}  \Y(i')
  =
    \left[
      \frac{1}{s} \binom{2s} {s-1}  - 1
       \right] \ca^s (s+1)! .
\end{equation}
\end{prop}

\begin{proof}
  By comparing Eq.~\eqref{E:RHS_P2_h} for the set \PTwo with
  Eq.~\eqref{E:RHS_P1_sum_of_col} for the set \POne,
  we see that
  the sum of coefficients $\coe_{i,j}$ for the set \PTwo
  for each column has the same factor
  \[
    \binom{s-1}{ i}
  \]
  as that for the set \POne.
  This factor is dependent on $i$.
  An extra factor
  \[
    \frac{1}{s} \binom{2s} {s-1}  - 1
  \]
  is independent of $i$.
  Hence the total contribution from the set \PTwo
  is that of the set \POne (Eq.~\eqref{E:RHS_P1_V_gen}) multiplied by this factor.
\end{proof}

The total contribution for the vertical identities from
all the recurrence identities from the sets
\POne and \PTwo
is the sum of Eq.~\eqref{E:RHS_P1_V_gen} and
Eq.~\eqref{E:RHS_P2_V_gen}.
By symmetry the total contribution from the vertical and horizontal identities
is twice this number.
Taking into account that LHS is also the sum of vertical and horizontal identities
with a factor of $2$ (the first identity of Eq.~\eqref{E:goal_lhs}),
this factor of $2$ cancels.
Hence
the total contribution for RHS is 
\begin{prop}[Overall RHS] \label{P:RHS_P1P2_V_gen}
\begin{equation} \label{E:RHS_P1P2_V_gen}
  RHS = \ca^s \binom{2s}{s}  s!.
\end{equation}
\end{prop}

For two-dimensional monomer-polymer model,
$\ca = 2$,
and the main Theorem~\ref{T:main} is proved.

\begin{comment}

\section{Generalization}

The main theorem can be generalized to the following result, the proof of which
will be discussed in another paper.
\begin{claim}
  For positive integers $p$ and $q$ which are not equal to $1$ at the same time,
  \[
    \sum_{i=0}^{(p+q)s} (-1)^i \binom{(p+q)s}{i} a_{n-pi, m-qi} = 0.
  \]
\end{claim}

\begin{remark}
The Q1 lower triangle is just  $\coe(i,j')$,
but what are those in upper triangle of Q2?
$(-1)^s c(s-j, s-i)$
are put in $(i, j+s)$ for H and $(j+s, i)$ for V.

Example:
For $s=5$,
take $i=1, j=3$.
Then take $c(5-3, 5-1) = c(2, 4) = 125$.
Then
\[
  (1, 3+5) = (3+5, 1) = (-1)^5 \times (5-3, 5-1).
\]
\end{remark}

Proof of Prop~\ref{P:RHS_P2_V_gen}
\begin{proof}
  By making a change of summation variable in the second sum $i' \rightarrow 2s -i'$,
  the two sum can be combined into a single sum
  
  \[
    \sum_{i'=0}^{s-1}  \Y(i')
    + \sum_{i'=s+1}^{2s}  \Y(i')
    =
    \sum_{i'=0}^{s-1}
    \coe(i',j') 
    \left(
      c(n-i)^s
      + \left( -c(n-2s+i) \right)^s
    \right) .
  \]

\end{proof}

\end{comment}

\appendix

\section{Some properties of the two-dimensional sequence} \label{A:seq}

In this Appendix we discuss some properties of the two-dimensional sequence.

First, we evaluate the double generating function of the sequence.
Define
\[
  \cdgf(s) = \sum_{i=0} \cgf_i(x) z^i = \sum_{i=0} \sum_{j=i+1} \coe_{i,j} x^j z^i. 
\]

\begin{prop}[Double generating function] \label{P:double_GF}
  The double generating function of the sequence $\coe_{i,j}$ is
\begin{equation} \label{E:double_GF}
  \cdgf(s) =
  -\frac{(1-xz)^s} {1-z}
  -\frac{x(1-xz)^{2s} } {(1-x)^{s+1}}
  +\frac{(1-xz)^{2s} }{ (1-x)^s (1-z)}  .
\end{equation}
\end{prop}

\begin{proof}
  The double generating function is obtained from the three parts of $\cgf_i(x)$ (Eq.~\ref{E:H}).

  The first part can be obtained by exchanging the order of sums
  and use geometric sum and binomial theorem to obtain
  \[
    \cdgf^{(1)}(s) = -\frac{(1-xz)^s} {1-z}.
  \]

  The second part can also be obtained by binomial theorem:
  \[
    \cdgf^{(2)}(s) = -\frac{x(1-xz)^{2s} } {(1-x)^{s+1}} .
  \]

  The third part involves a double sum.
  By exchanging the order of sums,
  the sum over $i$ has the following summand
  \[
    \F = (-1)^i
    (2s-i)
    \binom{2s}{ i}
    \binom{i}{ t}
    z^i .
  \]
  \Ze's method shows that the sum $\Sum(t)$ satisfies the following first order homogeneous recurrence
  \[
    -(2s-t-1)z \Sum(t) + (z-1)(t+1) \Sum(t+1) = 0,
  \]
  with the certificate
  \[
    \R = i-t .
  \]
  With initial value at $t=0$ as
  \[
    \Sum(0) = -2s(1-z)^{2s-1},
  \]
  the sum is evaluated as
  \[
    \Sum(t) = - (2s) z^t  (z-1)^{2s-t-1}
      \binom{2s-1} {t} .
  \]
  Then the sum of over $t$ can be evaluated by \Ze's method again with $s$ as the parameter.
  For the summand
  \[
    \F =
    (2s) z^t (z-1)^{2s-t-1}
    \binom{2s-1}{ t}
    \frac{1} { 2s-t}
    \frac{1} { (1-x)^{s-t}},
  \]
  the sum obeys the following first order homogeneous recurrence
  \[
    (xz-1)^2 \Sum(s)  + (x-1) \Sum(s+1) = 0,
  \]
  with the certificate
  \[
    \R =
    \frac{ t(z-1)(2xz+2zs-3+2zsx + z-4s-zxt+t)}
    {(2s+2-t)(2s-t+1)} .
  \]
  With initial value at $s=0$ as
  \[
    \Sum(0) = \frac{1}{1-z},
  \]
  the sum is evaluated as
  \[
    \Sum(s) =  \cdgf^{(3)}(s) = \frac{(1-xz)^{2s} }{ (1-x)^s (1-z)}  .
  \]
  
\end{proof}

This two-dimensional sequence has many interesting properties.
We point out a few such properties, which will not be used explicitly in the proof of
the main theorem.

First we point out that the sequence takes only integer values.
\begin{coro} [All items are integer]
  All items in the sequence are integers. 
\end{coro}
\begin{proof}
  This is obvious from the double generating function
Eq.~\eqref{E:double_GF}.
\end{proof}

From the explicit expressions (Eq.~\eqref{E:double_GF} or Eq.~\eqref{E:H}  or  Eq.~\eqref{E:hv2}),
a few simple expressions for the first few items can be obtained:
\begin{coro}
  The general expression of the first few terms of $ \coe_{s, i, j}$:
  \begin{align*}
    \coe_{s, 0, 1} &= s-1, 
    & \coe_{s, 0, 2} &= \frac{1}{2} (s+1)(s-2),
    & \coe_{s, 0, 3} &= \frac{1}{6}  (s+1)(s+2) (s-3), \\
                   &               &  \coe_{s, 1, 2} &= -\frac{1}{2} s(3s-5), 
    &    \coe_{s, 1, 3} &= -\frac{1}{6}s(s+1)(5s-14), \\
                   & & & & \coe_{s, 2, 3} &= \frac{1}{6} s(7s^2-21s+8) .                
    \end{align*}
\end{coro}

The sum of the column is given by
\begin{coro}[Sum of column]
  The column sum of the sequence is given by
  \[
    \sum_{i=0}^{j-1} \coe_{s,i,j} = (-1)^{j+1}
    (s-1)
    \binom{s-1} {j-1} .
  \]
\end{coro}
\begin{proof}
  Take limit of the double  generating function at $z=1$ to get 
  \[
    (s-1) (1-x)^{s-1} x .
  \]
\end{proof}

As a simple corollary of Eq.~\eqref{E:hv2},
the sequence always vanishes at the point $(s, 0, s)$.

\section{Gosper and Zeilberger algorithms}  \label{S:A_GZ}

In the proof we need to evaluate many sums, some are single sums, while others are
double sums. Two powerful methods to evaluate sums are \Go's algorithm and \Ze's algorithm.

\Go's algorithm is a procedure to find \emph{indefinite} sum of hypergeometric terms
\cite{gosperDecisionProcedureIndefinite1978,Graham1994,Petkovsek1996}.
Given a hypergeometric terms $t_n$,
the algorithm finds, if it exists, a hypergeometric term $z_n$, such that
\[
  z_{n+1} - z_{n} = t_n.
\]
The \anti{} $z_{n}$ for the indefinite sum  is analogous to the antiderivative for the indefinite
integration.
\Go's algorithm gives a definitive answer to the question
of whether a given hypergeometric term can be indefinitely summed.
If the answer is confirmative, we call the summand \Go-summable.

Many summands, however, are not \Go-summable.
\Ze's algorithm,
or the method of \emph{creative telescoping}
\cite{zeilbergerMethodCreativeTelescoping1991,Graham1994,Petkovsek1996},
makes many more summands become summable.
The method is an extension of \Go's algorithm.
It shows that  for the sum
  \begin{equation} \label{E:A_sum}
    f(n) = \sum_{k=u}^{v} \F(n, k),
  \end{equation}
it is always possible to find a recurrence
\begin{equation} \label{E:A_recur}
   \sum_{j=0}^J b_j(n) \F(n+j, k) = \G(n, k+1) - \G(n, k), 
\end{equation}
in which the coefficients $\{b_j(n)\}_0^J$ are polynomials in $n$, and 
$\R(n, k) = \G(n, k)/\F(n, k)$ is a rational function of $n$ and $k$.
In the proof we list $\R(n, k)$ for each sum,
which can be used to check the correctness
of the recurrence relation of the summand.

Since the coefficients $b_j(n)$ on the left side of the recurrence
are independent of the summation index $k$,
Eq.~\eqref{E:A_recur} can be summed from $u$ (the lower bound of the summation index $k$)
to $v$ (the upper bound) to
obtain the sum $f(n)$:
\[
  \sum_{j=0}^J b_j(n) f(n+j) = \G(n, u+1) - \G(n, v), 
\]
where the right side telescopes.
For the \emph{standard boundary conditions}~\cite{wilfAlgorithmicProofTheory1992},
where the sum includes all nonzero values of the summand
(and the support lies inside a larger set in which the summand remains well defined
and vanishes),
the right hand side of the above equation vanishes,
so the equation gives a homogeneous recurrence relation for the sum.
For nonstandard boundary conditions, however,
the right hand side may not vanish, leading to an inhomogeneous recurrence for the sum.
For the sums in our case,
they most often have nonstandard boundary conditions
due to the geometric restrictions imposed on the lower and upper boundaries,
hence they often have inhomogeneous recurrence relations.

The algorithm is extended to multi-variable
multi-sum summations
\cite{wilfAlgorithmicProofTheory1992,Apagodu2006}.
For a double sum with summation indices $i$ and $j$,
\begin{equation}
  \label{E:A_double_sum}
  f(n) =
  \sum_{i=u_i}^{v_i}
  \sum_{j=u_j}^{v_j}
  \F(n, i, j),
\end{equation}
the recurrence of Eq.~\eqref{E:A_recur}
is in the form of
\[
  \sum_{\ell=0}^J b_\ell(n) \F(n+\ell, i, j) =
  \Delta_i \G_i(n, i, j)
  +
  \Delta_j \G_j(n, i, j),
\]
where
$\Delta_i$ and $\Delta_j$
are the forward difference operators
on $i$ and $j$, respectively.

Just as the single sums mentioned above,
the double sums in our proof often
have nonstandard boundary conditions.
As the authors commented in~\cite{wilfAlgorithmicProofTheory1992},
under  nonstandard boundary conditions,
``the boundary terms `snowball' so the
application would become rapidly more complicated.''

Whether summation indices are dependent of each other
also affects the recurrence relation of the sum.
Below we list a few such scenarios for the summation indices
of double sums.
In each case
the inhomogeneous term is a sum of two sums with a single variable.
Here we use $u_\ell$ and $v_\ell$ for the lower and upper boundaries,
respectively, with $\ell=i$ or $j$ for the summation indices.
They are assumed to be independent of the summation parameter $n$.
The proofs can be obtained by direct calculations.

\begin{lemma} \label{L:inhomo1}
  For the double sum with independent lower and upper boundaries,
  \[
    \sum_{i=u_i}^{v_i}
    \sum_{j=u_j}^{v_j}
    \F(n, i, j),
  \]
  the sum satisfies an inhomogeneous recurrence relations
  with the inhomogeneous term as
  \begin{multline*}
   \sum_{i=u_i}^{v_i}
    \sum_{j=u_j}^{v_j}
    \big(
      \Delta_i G_i(n, i, j)
    +
    \Delta_j \G_j(n, i, j)
  \big)
  = \\
  \sum_{i=u_i}^{v_i}
  \big(
  \G_j(n, i, v_j+1)
  -
  \G_j(n, i, u_j)
  \big)
  +
  \sum_{j=u_j}^{v_j}
  \big(
  \G_i(n, v_i+1, j)
  -
  \G_i(n, u_i, j)
  \big)
  .
  \end{multline*}
\end{lemma}

\begin{lemma} \label{L:inhomo2}
  For the double sum with inner sum's upper boundary as the summation index of the outer sum,
  \[
    \sum_{i=u_i}^{v_i}
    \sum_{j=u_j}^{i}
    \F(n, i, j),
  \]
  the sum satisfies an inhomogeneous recurrence relations
  with the inhomogeneous term as
  \begin{multline*}
   \sum_{i=u_i}^{v_i}
    \sum_{j=u_j}^{i}
    \big(
    \Delta_i G_i(n, i, j)
    +
    \Delta_j \G_j(n, i, j)
    \big)
    =\\
    \sum_{i=u_j}^{v_i}
    \big(
    \G_j(n, i, i+1)
    -
    \G_j(n, i, v_j)
    \big)
    +
    \sum_{j=u_j}^{v_i}
     \big(
     \G_i(n, v_i+1, j)
     -
     \G_i(n, j, j)
     \big)
     .
  \end{multline*}
\end{lemma}

\begin{lemma}  \label{L:inhomo3}
  For the double sum
  \[
    \sum_{i=u_i}^{v_i}
    \sum_{j=u_j}^{v_i - i}
    \F(n, i, j),
  \]
  the sum satisfies an inhomogeneous recurrence relations
  with the inhomogeneous term as
  \begin{multline*}
   \sum_{i=u_i}^{v_i}
    \sum_{j=u_j}^{i}
    \big(
    \Delta_i G_i(n, i, j)
    +
    \Delta_j \G_j(n, i, j)
    \big)
    =\\
    \sum_{i=u_i}^{v_i}
    \big(
    \G_j(n, i, v_i+1 - i)
    -
    \G_j(n, i, u_j)
    \big)
    +
    \sum_{j=u_j}^{v_i}
     \big(
     \G_i(n, v_i+1-j, j)
     -
     \G_i(n, u_i, j)
     \big)
     .
  \end{multline*}
\end{lemma}

\begin{remark}
  The choice of summation parameter also affects the recurrence obtained
  by the \Ze's method.
  If one of the running indices of the double sum Eq.~\eqref{E:A_double_sum}
  is used as a parameter ($u_\ell$ or $v_\ell$),
  the inhomogeneous term needs to include additional single sums
  for the boundary terms.
\end{remark}

\begin{remark}
Note that in the proof we use $\F$ and $\Sum$ to denote the \emph{generic} summand and sum under discussion.
\end{remark}

\section{Some binomial coefficient identities} \label{S:A_binom}

One of the binomial coefficient identities that is used several times in the proof
is the well-known Chu-Vandermonde identity.
\begin{equation} \label{E:binom_iden_V}
  \sum_{j} \binom{a}{j} \binom{b}{c-j} = \binom{a+b}{c} . 
\end{equation}

\begin{comment}
\begin{equation} \label{E:binom_iden_prud_56}
  \sum_{k=0}^n (-1)\binom{n}{k} \binom{a-k}{m} = \binom{a-n}{m-n} . 
\end{equation}

\begin{equation}
  \binom{2s+1}{i} = \sum_{j=0}^s \binom{s}{j} \binom{s+1}{i-j} .
\end{equation}
  
\end{comment}

Another binomial coefficient identity used in the proof of the result in the first \qu
is listed below.

  \begin{lemma} \label{L:binom_iden}
    \begin{equation} \label{E:binom_iden}
    \sum_{j'=0}^{j} (-1)^{j'}  \binom{s-t+j'}{ j'}  \binom{s}{ j-j'} =
    (-1)^j  \binom{j-t}{j}.
    \end{equation}
  \end{lemma}
  \begin{proof}
    Rearrange the binomials in the summand,
    \begin{equation*} \label{E:binom_iden_L}
      \sum_{j'=0}^{j} (-1)^{j'}  \binom{s-t+j'}{ j'}  \binom{s}{ j-j'}
      = \frac{s! (j-t)! }{(s-t)! j!}
        \sum_{j'=0}^{j} (-1)^{j'}  \binom{j}{j'} \binom{s-t+k}{j-t} .
      \end{equation*}
      By using identity
      \begin{equation} \label{E:binom_iden0}
        \sum_{k=0}^n \binom{n}{k} \binom{a+k}{m} = (-1)^n \binom{a}{m-n},
      \end{equation}
      Eq.~\eqref{E:binom_iden} can be obtained.
      Eq.~\eqref{E:binom_iden0} can be found in many references,
      such as Prudnikov 4.2.5. Eq. (55)
      or Knuth 1.2.6. Eq. (23).
  \end{proof}

\begin{comment}
  
  From identity Eq.~\eqref{E:binom_iden} we have
  \begin{coro} \label{C:binom_iden}
    \begin{equation} \label{E:binom_iden_C}
      \sum_{j'=0}^{j} (-1)^{j'}  \binom{s-t+j'}{ j'}  \binom{s}{ j-j'} =
      \begin{cases}
        (-1)^j (j+1), & t=-1, \\
        (-1)^j , & t=0, \\
        1 ,      & t=j, \\
        0 ,      & 0 < t < j .
      \end{cases}
    \end{equation}
  \end{coro}
  
\begin{lemma} \label{L:i0}
  For $0 \le j \le s$,
  \[
    \sum_{j'=0}^j (-1)^{j'} \binom{s+j'-1}{ j' } \binom{s}{j-j'} \frac{s-j'}{s} =
    \begin{cases}
      1 & j = 0, \\
      (-1)^{j+1} & j > 0 .
    \end{cases}  
  \]
\end{lemma}
\begin{proof}
  Denote the sum as $S(j)$.
  By direct calculations it is easy to show that
  the $S(0)=1$ and $S(1)=1$.
  WZ method shows that
  \[
    S(j+1) = -S(j).
  \]
\end{proof}

\begin{lemma}
  For $0 \le j \le s$,
  \[
    (-1)^j \sum_{j'=1}^j (-1)^{j'} \binom{s+j'-1}{ j' } \binom{s}{j-j'} \frac{s-j'}{s} =
    \begin{cases}
      0 & j = 0, \\
      (-1)^{j+1} \binom{s}{j}  - 1 & j > 0 .
    \end{cases}  
  \]
\end{lemma}
\begin{proof}
This is proved by using Lemma~\ref{L:i0}.
\end{proof}   

\end{comment}

\section{Some sums used in section~\ref{S:RHS} that are \Go-summable} \label{A:RHS_Gsum}
  
Below we list some summands that appear several times in section~\ref{S:RHS}
that are \Go-summable.
\begin{lemma} \label{L:anti1}
  For the summand 
  \[
    \F =
    \binom{s+b+j'} { j'},
  \]
  with $j'$ as the index of summation,
  the \anti is
  \[
    \binom{s+b+j'} {j'-1} .
  \]
\end{lemma}

\begin{lemma} \label{L:anti2}
  For the summand
  \[
    \F =
    (-1)^{j'}
    \binom{s}{j'} ,
  \]
  with $j'$ as the index of summation,
  the \anti is
  \[
    (-1)^{j'} \binom{s-1}{j'-1} .
  \]
\end{lemma}

\bibliographystyle{jabbrv_abbrv} % here load jabbrv.bst
%\bibliography{../../kmer_full_jname_2023_09_17.bib}
%\bibliography{kmer_full_jname_2023_09_17.bib}

\end{document}